\documentclass[12pt]{amsart}
\usepackage{hyperref}

\newtheorem{theorem}{Theorem}[section]

\newtheorem{lemma}[theorem]{Lemma}
\newtheorem{proposition}[theorem]{Proposition}

\theoremstyle{theorem}

\theoremstyle{definition}
\newtheorem{definition}[theorem]{Definition}

\theoremstyle{remark}

\numberwithin{equation}{section}

\newcommand{\p}{\partial}
\newcommand{\n}{\nabla}

\newcommand{\fder}[2]{\frac{\partial #1}{\partial #2}}

\newcommand{\la}{\langle}
\newcommand{\ra}{\rangle}

\begin{document}
\title{Free boundary minimal surfaces in products of balls}
\author{Jaigyoung Choe}
\address{Korea Institute for Advanced Study, Seoul, 02455, Korea}
\email{choe@kias.re.kr}
\author{Ailana Fraser}
\address{Department of Mathematics \\
                 University of British Columbia \\
                 Vancouver, BC V6T 1Z2}
\email{afraser@math.ubc.ca}
\author{Richard Schoen}
\address{Department of Mathematics \\
                 Stanford University \\
                 Stanford, CA 94305 
                 \& Department of Mathematics \\
                 University of California \\
                 Irvine, CA 92617}
\email{schoen@stanford.edu}
\thanks{2020 {\em Mathematics Subject Classification.} 53A10, 58E12, 58C40. \\
J. Choe was supported in part by NRF-RS-2023-00246133,
A. Fraser was supported in part by  the 
Natural Sciences and Engineering Research Council of Canada, R. Schoen
was supported in part by NSF grant DMS-2005431. This material is based upon work supported by the National Science Foundation under Grant No. DMS-1928930, while the authors were in residence at the Simons Laufer Mathematical Sciences Institute (formerly MSRI) in Berkeley, California, during the Fall 2024 semester.}

\begin{abstract} In this paper we develop an extremal eigenvalue approach to the problem of construction of free boundary minimal surfaces in the product
of Euclidean balls of chosen radii. The extremal problem involves a linear combination of normalized mixed Steklov-Neumann eigenvalues. The problem is motivated by the Schwarz $P$-surface which is a free boundary minimal surface in a cube. We show that the problem often does not have an absolute maximum in the product case even though it is bounded from above. By imposing a finite group of symmetries on both the surface and on the eigenfunctions we construct at least one free 
boundary minimal surface in a rectangular prism with arbitrary side lengths. We further show that unless the rectangular prism is a cube there are at least two such surfaces. We also prove that any immersed free boundary minimal surface of genus $0$ with one boundary component on each face
of a rectangular prism, and that is invariant under the reflections interchanging opposite faces of the prism, is necessarily embedded. Finally we show that for a genus 
$0$ surface with $6$ boundary components and suitable reflection symmetries there is a maximizing metric which can be realized by a free boundary minimal immersion into a product of Euclidean balls.  
\end{abstract}

\maketitle

\section{Introduction} There has been a great deal of success in recent years in the construction and geometric understanding of free boundary 
minimal surfaces in Euclidean balls (see for example \cite{FS1}, \cite{FS2}, \cite{Li}, \cite{P1}, \cite{P2}, \cite{P3}, \cite{KKMS}). In this paper we develop an analogous theory for free boundary minimal surfaces in products
of Euclidean balls of arbitrarily chosen radii. The work is partially motivated by the Schwarz $P$-surface which is a genus $0$ free boundary minimal
surface with $6$ boundary components, one on each face of a cube in $\mathbb R^3$. One can ask if such a surface has an eigenvalue characterization,
and whether it can be generalized to free boundary surfaces in arbitrary rectangular prisms (products of three intervals with arbitrary lengths). 

In the first part of the paper we develop a general theory which relates particular eigenvalue problems to free boundary minimal surfaces in products of
balls. We begin with a surface $M$ with a finite number of boundary components. We divide the boundary components into $k$ nonempty disjoint
collections of components denoted $\Gamma_1,\Gamma_2,\ldots, \Gamma_k$. Thus each $\Gamma_i$ is a nonempty collection of boundary components
with $\Gamma_i\cap\Gamma_j=\phi$ for $i\neq j$, and each boundary component of $M$ lies in some $\Gamma_i$. For each $i$ we consider an
eigenvalue problem $(SN_i)$ with Steklov boundary condition on the boundary components of $\Gamma_i$ and Neumann boundary condition on the other
boundary components. For a given metric $g$ on $M$ we have a lowest nonzero first eigenvalue $\sigma_1^{(i)}(g)>0$ for $i=1,\ldots,k$. Given positive
numbers $a_1,a_2,\ldots, a_k$ we then define a function
\[ F(g)=\sum_{i=1}^ka_i^2L_i(g)\sigma_1^{(i)}(g)
\]
where $L_i(g)$ denotes the total length of the boundary components in $\Gamma_i$ with respect to $g$. We show that if $k\geq 2$ then although
$F(g)$ is bounded from above, its maximum is not realized by a smooth metric. 

For this reason we consider metrics that are invariant under a finite group $G$ of diffeomorphisms of $M$. We prove that if $g$ maximizes $F$
among $G$-invariant metrics, then $g$ is realized on an equivariant free boundary minimal immersion to a product of balls 
$B^{n_1}(a_1)\times\cdots\times B^{n_k}(a_k)$ of radii $a_i$ and some dimensions $n_i$. 

The existence then comes down to the choice of surface $M$, the $\Gamma_i$, and the group $G$ of diffeomorphisms. In the remainder of the 
paper we consider the case in which $M$ is a genus $0$ surface with $6$ boundary components where we group them into three sets of two each 
$\Gamma_1,\Gamma_2,\Gamma_3$. We consider three commuting reflection symmetries $\rho_1,\rho_2,\rho_3$ where $\rho_i$ interchanges the two boundary components of $\Gamma_i$ and preserves the other boundary components. A topological model for such a surface would be the domain in $\mathbb S^2$
gotten by removing disks of equal radii centered on the positive and negative coordinate axes with symmetries being reflection across the coordinate
planes. In the remainder of the paper we do a careful analysis of the extremal problem for $F$ among $G$-invariant metrics. 

We prove two main existence results. First we show that for any choices of $a_1,a_2,a_3$ there is a free boundary minimal immersion of $M$ to
the rectangular prism with side lengths $a_1,a_2,a_3$. We obtain this result by restricting to eigenfunctions for $(SN_i)$ that are odd with respect to
$\rho_i$ and even with respect to $\rho_j$ for $j\neq i$. Thus we take
\[ F^o(g)=\sum_{i=1}^ka_i^2L_i(g)\sigma_1^{(io)}(g)
\]
where $\sigma_1^{(io)}$ denotes the first odd eigenvalue of $(SN_i)$. We describe explicitly the maximizer of $F^o$ in a conformal class, denoting this
value by $\bar{F}^o(p)$ where $p$ is a point in the moduli space of conformal structures. 
We obtain the minimal immersion as a critical point of the function $\bar{F}^o$ on the moduli space of $G$-invariant conformal structures. This critical 
point is not a global maximum; in fact, $\bar{F}^o$ is not bounded from above on the moduli space. 

We go on to show that, unless the rectangular prism is a cube, there are at least two such free boundary surfaces, and that each one is embedded. 
The multiplicity proof exploits the non-invariance of the energy functional under the conformal symmetry group of the surface. The proof of
embeddedness applies to any free boundary minimal immersion of a genus $0$ surface to a rectangular prism having one boundary component
on each face of the prism and that is invariant under the reflections interchanging opposite faces of the prism. 

Finally we show that $F$ achieves its maximum on the space of $G$-invariant metrics on $M$. This produces a minimal immersion into a product of
three balls $\Pi_{i=1}^3\mathbb B^3(a_i)$ contained in $\mathbb R^9$. We do not expect this surface to lie in $\mathbb R^3$ in general.  

To put the paper in a broader context and to include references to earlier work, we make some comments about other possible approaches. First,
we know that free boundary minimal surfaces in a domain are those that are stationary with respect to variations that preserve the domain,
but do not necessarily fix the boundary pointwise.  While the surfaces we construct do have this variational property, it seems difficult to construct
them from min-max methods because the surfaces are also required to miss the lower dimensional skeletons of the boundary of the rectangular
prism (or more generally the singular part of the boundary of the product domain). It is not clear whether there is a min-max approach that could
accomplish this. Secondly we observe that the surfaces we construct in the three dimensional case extend to triply periodic minimal surfaces in
$\mathbb R^3$. An extensive existence theory for triply periodic minimal surfaces has been developed by W. Meeks \cite{Me}. The surfaces of
\cite{Me} do not generally have fundamental domains in which they are free boundary minimal surfaces. It is also natural to attempt to construct
free boundary minimal surfaces in general rectangular prisms by deforming the cube and considering how the Schwarz P-surface deforms. Such
an approach involves understanding the null space of the Jacobi operator on the Schwarz P-surface and involves a delicate perturbation theory.
This approach was considered by M. Koiso, P. Piccione, and T. Shoda \cite{KPS} for some triply periodic minimal surfaces including the Schwarz
P-surface. It might be possible to use this result to construct our surfaces for rectangular prisms with nearly equal sides, but it seems difficult to handle
large deformations of the prism since it would be necessary to show that the surfaces do not degenerate when the sides are deformed by
a large amount. Finally we point out the construction of minimal surfaces from discrete approximations developed by A. Bobenko, T. Hoffman,
and B. Springborn \cite{BHS}. The paper contains a construction of the Schwarz P-surface and many others. It is interesting to see if such
a construction can be extended to the general case of rectangular prisms. 

\section{Preliminaries: mixed Steklov-Neumann problem} \label{section:preliminaries}

Let $(M,g)$  be a Riemannian manifold with boundary  $\partial  M$.   Assume that  $\partial M$  is  the disjoint union of $\mathcal{F}$ and $\mathcal{B}$. Consider the eigenvalue problem 
\begin{equation} \label{equation:steklov-neumann}
\begin{cases}
        \Delta_g  f =0   & \mbox{  in } M, \\
        \nu f =0  & \mbox{ on } \mathcal{B}, \\
        \nu f =\sigma f  & \mbox{  on } \mathcal{F}
\end{cases}
\end{equation}
where $\nu$ is the outward unit conormal along  $\partial M$.
When $\mathcal{F}$ is empty we obtain the Neumann problem, and when $\mathcal{B}$ is empty we obtain the Steklov problem.
When $\mathcal{B}$  and $\mathcal{F}$  are both nonempty we obtain  the {\em mixed Steklov-Neumann problem} (\ref{equation:steklov-neumann}),  which has a discrete set of eigenvalues
\[
       0=\sigma_0 < \sigma_1 \leq  \sigma_2 \leq \cdots \leq \sigma_k \leq \cdots \rightarrow \infty
\]
and each eigenvalue has finite multiplicity (see for example \cite[Chapter III]{B}). 
The first  nonzero eigenvalue $\sigma_1$  has the variational characterization
\[
      \sigma_1=\inf_{\int_{\mathcal{F}} u=0}
       \frac{\int_M \|\n u\|^2 \,dv_M}{\int_{\mathcal{F}} u^2 \,dv_{\partial M}}
\]
and in general,
\[
        \sigma_k= \inf \left\{\frac{\int_M \|\n u\|^2 dv_M}{\int_{\mathcal{F}} u^2 \, dv_{\p M}} \; : \; 
        \int_{\mathcal{F}} u \phi_j=0 \mbox{ for } j=0, 1, 2, \ldots, k-1 \; \right\}
\]      
where $\phi_j$ is an eigenfunction corresponding to the eigenvalue $\sigma_j$, for $j=1, \ldots, k-1$.
Alternatively,
\[ \label{equation:minmax}
     \sigma_k=\inf_{E \in \mathcal{E}(k+1)} \sup_{0 \neq u \in E} 
     \frac{\int_M \|\n u\|^2 \,dv_M}{\int_{\mathcal{F}} u^2 \,dv_{\partial M}}
\]
where $\mathcal{E}(k+1)$ is the set of all $(k+1)$-dimensional subspaces of $W^{1,2}(M)$.

The mixed Steklov-Neumann problem can  be viewed as a special case of the weighted  Steklov problem
\[
\begin{cases}
      \Delta_g f=0 & \mbox{  on } M \\
      \nu f =\sigma \rho f & \mbox{ on } \partial M
\end{cases}
\]
where $\rho \in L^\infty$ is a non-negative nonzero function, when $\rho$ is zero on $\mathcal{B}$, or as a  special case of the general setting of eigenvalues on measure spaces with measures supported on the boundary (that is, with boundary volume measures) as in \cite{K}. In particular, the same eigenvalue bounds that  hold in these settings hold for the mixed Steklov-Neumann eigenvalues. For surfaces we have:
\begin{proposition}[\cite{K}, \cite{FS1}, \cite{Ka}]
Let $(M,g)$ be a Riemannian surface of genus $\gamma$ with $b$  boundary components, then
\[
     \sigma_k(g) L_g(\mathcal{F}) \leq \min \{ C(\gamma+1) k, 2\pi (k+\gamma+b-1) \},
\]
where $C$ is independent of the metric $g$.
\end{proposition}

\section{Eigenvalue characterization of free boundary minimal submanifolds in products of balls}

In this section  we establish a connection between the Steklov-Neumann eigenvalue problem and free boundary minimal surfaces in rectangular prisms, and more generally, products of balls in $\mathbb{R}^n$.

\begin{lemma} \label{lemma:characterization}
Let $\Sigma^m$ be a properly immersed submanifold in a rectangular prism $P= [ -a_1,a_1] \times [-a_2, a_2] \times \cdots \times [-a_n,a_n]$ in $\mathbb{R}^n$, with each component of  the boundary $\partial \Sigma$ lying in a face 
$F^{\pm}_i=(-a_1, a_1) \times \cdots \times (-a_{i-1}, a_{i-1}) \times \{x_i=\pm a_i\} \times (-a_{i+1}, a_{i+1}) \times \cdots \times  (-a_n, a_n)$ 
of the prism, for  some $i$, $1 \leq i \leq n$. Then $\Sigma^m$ is a minimal submanifold that  meets  $\partial P$ orthogonally if and only  if the coordinate  functions of $\Sigma$ in $\mathbb{R}^n$ are Steklov-Neumann eigenfunctions:
\[
\begin{cases}
        \Delta_{\Sigma}  x_i =0   & \mbox{  on } \Sigma, \\
        \nu x_i=0  & \mbox{ on } \partial \Sigma \cap F_j^\pm, \;  j=1, \ldots, n, \, j\ne i, \\
        \nu  x_i =\frac{1}{a_i} x_i  & \mbox{  on } \partial \Sigma  \cap F_i^\pm
\end{cases}
\]
for $i=1, \ldots, n$.
\end{lemma}
\begin{proof}
Let $\Sigma$ be a properly immersed surface in a rectangular prism $P$ as in the statement.  

If $x=(x_1, \ldots, x_n)$ is the position vector, then it is well known  that $\Delta_\Sigma  x = H$, where $H$ is the mean curvature of $\Sigma$ in  $\mathbb{R}^n$, and $\Sigma$ is minimal if and only if the coordinate functions $x_i$ of $\Sigma$ in $\mathbb{R}^{n}$ are harmonic functions, $\Delta_\Sigma x_i=0$ on $\Sigma$ for $i=1, \ldots, n$. 

If $\nu$ is the outward unit conormal along $\p \Sigma$, then $\Sigma$ meets $\partial P$ orthogonally  if and only if $\nu =\pm e_i$ on $\partial \Sigma \cap F_i^\pm$, $i=1, \ldots, n$,  where $e_i$ denotes the $i$-th standard basis vector of $\mathbb{R}^n$, or equivalently
\[
            \nu x=\nu = \pm e_i 
            = (0, \ldots, 0 , \tfrac{1}{a_i} x_i, 0, \ldots, 0) \mbox{ on } \partial \Sigma \cap F_i^\pm,
\]
i.e.  $\nu x_i=\frac{1}{a_i} x_i$ and $\nu x_j=0$ for $j \ne i, \, 1 \leq j \leq n$ on $\partial \Sigma \cap F_i^\pm$.
\end{proof}

For the remainder of this section we suppose that 
$(M,  g)$ is a two-dimensional Riemannian manifold with boundary $\partial M = \Gamma_1 \cup \cdots \cup \Gamma_k$ , with $\Gamma_i \neq  \emptyset$ and $\Gamma_i \cap \Gamma_j = \emptyset$ for $i \ne j$,  $1 \leq i, j\leq k$. For each $i=1, \ldots, k$, consider the mixed Steklov-Neumann eigenvalue problem 
\[
(SN_i)  \quad
\begin{cases}
        \Delta_g  f =0   & \mbox{  in } M, \\
        \nu f =\sigma^{(i)} f  & \mbox{ on } \Gamma_i, \\
        \nu f =0  & \mbox{  on } \partial M \setminus \Gamma_i
\end{cases}
\]
with spectrum denoted by
\[
       0=\sigma_0^{(i)} < \sigma_1^{(i)} \leq  \sigma_2^{(i)} \leq  \cdots \rightarrow \infty.
\]
Let $V^{(i)}_j$ denote the eigenspace corresponding to $\sigma_j^{(i)}$.

If  $M$ is immersed, with induced metric $g$, as  a free boundary minimal surface in a rectangular prism in $\mathbb{R}^k$, with $\Gamma_i \subset F_i^\pm$ for $i = 1, \ldots,  k$  then by Lemma \ref{lemma:characterization}, we have
\begin{align*}
      2 \mbox{Area}(M) &  = \int_{M} \|\nabla x\|^2 \\
      &= \int_{\partial M} x \cdot \nu x  \\
      &= \int_{\Gamma_1} x \cdot \nu x + \cdots + \int_{\Gamma_k} x \cdot \nu x \\
      &= \int_{\Gamma_1} \sigma^{(1)} x_{1}^2 + \cdots + \int_{\Gamma_k} \sigma^{(k)} x_{k}^2  \\    
      &= a_{1}^2\sigma^{(1)} L(\Gamma_1) + \cdots + a_{k}^2 \sigma^{(k)}  L (\Gamma_k)
\end{align*}
where $\sigma^{(i)}=1/a_i$ is a Steklov-Neumann eigenvalue for ($SN_i$).
Since $M$ is a free boundary minimal surface, the first variation of area is zero for all admissible variations. This suggests that $(M,g)$ is `critical' with respect to variations of the metric for the functional
\[
  F(g):=  a_{1}^2 \, \sigma^{(1)}(g) L_g(\Gamma_1) + \cdots + a_{k}^2 \, \sigma^{(k)}(g)  L_g (\Gamma_k).
\]
Observe that the functional $F$ satisfies  a natural conformal invariance.
\begin{definition}
We say that two metrics $g_1$ and $g_2$ on a surface $M$  are {\em $F$-homothetic} (resp. {\em $F$-isometric}) if there is a conformal diffeomorphism $\varphi:M\to M$ that is a homothety (resp. an isometry) on each part $\Gamma_i$ of the boundary; that is, $\varphi^*g_2=\lambda^2 g_1$ and $\lambda=c_i$ on $\Gamma_i$ for some positive constant $c_i$ for $i=1, \ldots, k$
(resp. $\lambda=1$ on $\p M$). 
\end{definition}
Note that,
\begin{itemize}
\item If $g_1$ and $g_2$ are $F$-isometric then they have the same Steklov-Neumann
eigenvalues.
\item If $g_1$ and $g_2$ are $F$-homothetic then they have the same normalized
Steklov-Neumann eigenvalues; that is, $\sigma^{(i)}_j(g_1)L_{g_1}(\Gamma_i)=\sigma^{(i)}_j(g_2)L_{g_2}(\Gamma_i)$ for $i=1, \ldots, k$ and each $j$.
\end{itemize}

We now show that any metric that maximizes the functional $F$ among smooth metrics on the surface $M$, is $F$-homothetic to the induced metric from a free boundary minimal immersion of $M$ by first eigenfunctions into a product of balls. We will see later that in general maximizing metrics do not exist unless we impose symmetry and so we will characterize maximizing metrics more generally in the setting of equivariant optimization.

\begin{theorem} \label{theorem:characterization}
Let $G$ be a finite group acting on the surface $M$ by diffeomorphisms.
Suppose there exists a $G$-invariant metric $g_0$ on $M$ such that $F(g_0) = \sup_g F(g)$, where the supremum is over all smooth metrics on $M$ that are invariant under the action of $G$ on $M$. Then there exist independent first eigenfunctions $u^{(i)}_1, \ldots, u^{(i)}_{n_i} \in V_1^{(i)}$ for $i=1, \ldots, k$ and an orthogonal representation $\rho: G \rightarrow O(n_1) \times \cdots \times O(n_k)$ such that 
\[
     u:=(u^{(1)}_1, \ldots, u^{(1)}_{n_1}, \ldots, u^{(k)}_1, \ldots, u^{(k)}_{n_k}): M 
     \rightarrow B^{n_1}(a_1) \times \cdots \times B^{n_k}(a_k)
\]
is a $\rho$-equivariant conformal branched  free boundary minimal immersion with $u(\Gamma_i)\subset \partial B^{n_i}(a_i)$, where $B^{n_i}(a_i)$ is  the ball of radius $a_i$ centered at the origin in $\mathbb{R}^{n_i}$, and $g_0$ is $F$-homothetic to the induced metric from $u$ on $M$.  
\end{theorem}

In particular, if the multiplicity of the first Steklov-Neumann eigenspaces is one, then the maximizing metric is the induced metric from a free boundary (branched) minimal immersion of $M$ into a rectangular prism.
The proof is an extension of the proof of \cite[Proposition 5.2]{FS2} to the setting involving the Steklov-Neumann problem and the functional $F$, and adapted to the setting of maximizing among metrics with symmetries as in \cite[Theorem 2.12]{KKMS}.  

\begin{proof}
Let $g(t)$ be a smooth path of $G$-invariant metrics on $M$ with $g(0)=g_0$ and $\frac{d}{dt}g(t) =h(t)$, where  
$h(t)\in S^2(M)$ is a smooth path of $G$-invariant symmetric $(0,2)$-tensor fields on $M$.
Denote by $V_{g(t)}^{(i)}$ the eigenspace associated to the first nonzero Steklov-Neumann eigenvalue $\sigma^{(i)}_1(t)$ of $(M,g(t))$. 
It follows from the arguments of \cite[Lemma 5.1]{FS2} that $\sigma^{(i)}_1(t)$ is a Lipschitz function of $t$, and if $\dot{\sigma}_1^{(i)}(t)$ exists, then
\[
       \dot{\sigma}_1^{(i)}(t) = -\int_M \la \tau(u) , h \ra \; dA_t 
         -\tfrac{\sigma_1^{(i)}(t)}{2} \int_{\Gamma_i} u^2 h(T,T) \; ds_t,
\]
for any $u \in V_{g(t)}^{(i)}$  with $\|u\|_{L^2(\Gamma_i)}=1$, where $T$ is the unit tangent to $\p M$ for the metric $g(t)$, and
where $\tau(u)$ is the stress-energy tensor of $u$ with respect to the metric $g(t)$,
\[
    \tau(u)= du \otimes du -\tfrac{1}{2} \|\n u\|^2 g.
\]
Therefore $F(t):=F(g(t))$ is also a Lipschitz function of $t$, and at a point where the derivative exists,
\[
     \dot{F}(t)
       := Q_h(u^{(1)}, \ldots, u^{(k)})     
\]
where  
\[   
     Q_h(u^{(1)}, \ldots, u^{(k)})     =-\int_M \la \, \sum_{i=1}^{k } a_i^2 L_i \,  \tau(u^{(i)}), h \, \ra \; dA_t
       - \sum_{i=1}^k \int_{\Gamma_i} \tfrac{\sigma^{(i)}_1}{2} a_i^2  (L_i \,  (u^{(i)})^2-1) \, h(T,T)  \; ds_t 
\]
for any $u^{(i)} \in V_{g(t)}^{(i)}$ with $\|u^{(i)}\|_{L^2(\Gamma_i)}=1$, $i=1, \ldots, k$, where $L_i=L(\Gamma_i)$.

Consider the Hilbert space $\mathcal{H}=L^2(S^2(M))\times L^2(\Gamma_1) \times  \cdots \times L^2(\Gamma_k)$, where $L^2(S^2(M))$ denotes the space of $L^2$ symmetric $(0,2)$-tensor fields on $M$. 
Let $\mathcal{H}=L^2(S^2(M))^G \times L^2(\Gamma_1)^G \times  \cdots \times L^2(\Gamma_k)^G$
where $L^2(S^2(M))^G$ denotes the space of $L^2$ symmetric $(0,2)$-tensor fields on $M$ that are $G$-invariant and $L^2(\Gamma_i)^G$ the space of $G$-invariant functions on $\Gamma_i$. Let $\pi^G: \mathcal{H} \rightarrow \mathcal{H}^G$ be the orthogonal projection,
\[
      \pi^G(\omega, f_1, \ldots, f_k)
      =\frac{1}{|G|} \sum_{\sigma \in G} (\sigma^* \omega, \sigma^* f_1, \ldots, \sigma^* f_k).
\]

\begin{lemma} \label{lemma:indefinite}
Assume $g_0$ satisfies the conditions of Theorem \ref{theorem:characterization}. For any $(\omega,f_1, \ldots, f_k) \in \mathcal{H}^G$ 
there exist $u^{(i)} \in V^{(i)}_{g_0}$ with $\|u^{(i)}\|_{L^2(\Gamma_i)}=1$, $i=1, \ldots, k$, such that 
$$\la (\omega,f_1, \ldots, f_k), ( \sum_{i=1}^{k } a_i^2 L_i \,  \tau(u^{(i)}) , \tfrac{\sigma^{(1)}_1}{2} a_1^2  (L_1 \,  (u^{(1)})^2-1) , \ldots, \tfrac{\sigma^{(k)}_1}{2} a_k^2  (L_k \,  (u^{(k)})^2-1) ) \ra_{\mathcal{H}}=0.$$
\end{lemma}
\begin{proof}
Let $(\omega,f_1, \ldots, f_k) \in \mathcal{H}^G$. 
Since $C^\infty(S^2(M)) \times C^\infty(\Gamma_1) \times \cdots \times C^\infty(\Gamma_k)$ is dense in $L^2(S^2(M)) \times L^2(\Gamma_1) \times \cdots \times L^2(\Gamma_k)$ and $G$ is finite, we can approximate $(\omega, f_1, \dots, f_k)$ arbitrarily closely in $L^2$ by smooth $G$-invariant $(h,\tilde{f}_1, \ldots, \tilde{f}_k)$. 
We may redefine $h$ in a neighborhood of the boundary to a smooth $G$-invariant tensor whose restriction to $\Gamma_i$ is equal to the function $\tilde{f}_i$ for $i=1, \ldots, k$, and such that the change in the $L^2$ norm is arbitrarily small. In this way, we obtain a sequence of smooth $G$-invariant $h_n$ 
such that $(h_n, h_n|_{\Gamma_1}(T,T), \ldots, h_n|_{\Gamma_k}(T,T)) \rightarrow (\omega,f_1, \ldots, f_k)$ in $L^2$.

Let $g_n(t)=g_0+ th_n$. 
Then $g_n(0)=g_0$  and $\frac{dg_n}{dt}\big|_{t=0} =h_n$.
Given any $\epsilon >0$, by the fundamental theorem of calculus, 
\[
     \int_{-\epsilon}^0 \dot{F}(t) \; dt = F(0) - F(-\epsilon) \geq 0
\]
by the assumption on $g_0$. Therefore there exists $t$, $-\epsilon < t < 0$, such that 
$\dot{F}(t)$ exists and $\dot{F}(t) \geq 0$. Let $t_j$ be a sequence of points with $t_j<0$ and $t_j \rightarrow 0$, such that $\dot{F}(t_j)$  exists and $\dot{F}(t_j) \geq 0$. Choose $u_j^{(i)} \in V^{(i)}_{g_n(t_j)}$ with 
$||u_j^{(i)}||_{L^2(\Gamma_i, g_n(t_j))}=1$. Elliptic boundary estimates (\cite{M}) give bounds on $u_j^{(i)}$ and its derivatives up to the boundary, thus after passing to a subsequence $u_j^{(i)}$ converges in $C^2(M)$ to an eigenfunction $u^{(i)-}_n \in V_{g_n}^{(i)}$ with $||u^{(i)-}_n||_{L^2(\Gamma_i,g_n)}=1$.
Since $Q_{h_n}(u_j^{(1)}, \ldots, u_j^{(k)})=\dot{F}(t_j) \geq 0$, it follows that $Q_{h_n}(u^{(i)-}_n) \geq 0$. By a similar argument, taking a limit from the right, there exists $u^{(i)+}_n \in V_{g_n}^{(i)}$ with $||u^{(i)+}_n||_{L^2(\Gamma_i,g_n)}=1$, such that $Q_{h_n}(u^{(i)+}_n) \leq 0$.

As above, after passing to subsequences, $u^{(i)+}_n \rightarrow u^{(i)+}$ and $u^{(i)-}_n \rightarrow u^{(i)-}$ in $C^2(M)$, and
\begin{align*}
     \la &(\omega,f_1, \ldots, f_k),  ( \sum_{i=1}^{k } a_i^2 L_i \,  \tau(u^{(i)+}) , \tfrac{\sigma^{(1)}_1}{2} a_1^2  (L_1 \,  (u^{(1)+})^2-1) , \ldots, \tfrac{\sigma^{(k)}_1}{2} a_k^2  (L_k \,  (u^{(k)+})^2-1) ) \ra_{\mathcal{H}}   \\
     &= \lim_{n \rightarrow \infty} Q_{h_n}(u^{(1)+}_n,  \ldots, u^{(k)+}_n)\leq 0 \\
     \la &(\omega,f_1, \ldots, f_k),  ( \sum_{i=1}^{k } a_i^2 L_i \,  \tau(u^{(i)+}) , \tfrac{\sigma^{(1)}_1}{2} a_1^2  (L_1 \,  (u^{(1)-})^2-1) , \ldots, \tfrac{\sigma^{(k)}_1}{2} a_k^2  (L_k \,  (u^{(k)-})^2-1) ) \ra_{\mathcal{H}}   \\
     &= \lim_{n \rightarrow \infty} Q_{h_n}(u^{(1)-}_n,  \ldots, u^{(k)-}_n)\geq 0. \\
\end{align*}
\end{proof}
Let $K$ be the convex hull in $\mathcal{H}$ of 
\begin{align*}
        \{ \, & ( \sum_{i=1}^{k }  a_i^2 L_i \,  \tau(u^{(i)}) , \tfrac{\sigma^{(1)}_1}{2} a_1^2  (L_1\,  (u^{(1)})^2-1) , \ldots, \tfrac{\sigma^{(k)}_1}{2} a_k^2  (L_k \,  (u^{(k)})^2-1) ) \; : \\
        & \qquad u^{(i)} \in V^{(i)}_{g_0}, \, \|u^{(i)}\|_{L^2(\Gamma_i)}=1, \, i=1, \ldots, k\},
\end{align*}
and let $K^G=\pi^G(K)$.
We claim that $(0,0, \ldots, 0) \in K^G$.
If $(0,0, \ldots, 0) \notin K^G$, then since $K^G$ lies in a convex cone  in a finite dimensional subspace, the Hahn-Banach theorem implies the existence of $(\omega,f_1, \ldots, f_k) \in \mathcal{H}^G$ such that
\[
    \la (\omega,f_1, \ldots, f_k), ( \sum_{i=1}^{k } a_i^2 L_i \,  \tau(u^{(i)}) , \tfrac{\sigma^{(1)}_1}{2} a_1^2  (L_1 \,  (u^{(1)})^2-1) , \ldots, \tfrac{\sigma^{(k)}_1}{2} a_k^2  (L_k \,  (u^{(k)})^2-1) ) \ra_{\mathcal{H}}<0
 \]
 for all $u^{(i)} \in V^{(i)}_{g_0}, \, \|u^{(i)}\|_{L^2(\Gamma_i)}=1, \, i=1, \ldots, k$, which contradicts Lemma \ref{lemma:indefinite}. 
 
Therefore, $(0,0, \ldots, 0) \in K^G$, and since $K^G$ is contained in a finite dimensional subspace, there exist eigenfunctions 
$u_1^{(i)}, \ldots, u_{n}^{(i)} \in V_{g_0}^{(i)}$, $i=1, \ldots, k$, and positive real numbers $t_1, \ldots t_n$ with $\sum_{j=1}^n t_j=1$ such that
\[
   (0, \ldots, 0) =  \sum_{j=1}^n t_j \, \pi^G
   \left( \sum_{i=1}^{k } a_i^2 L_i \,  \tau(u_j^{(i)}) , \tfrac{\sigma^{(1)}_1}{2} a_1^2  (L_1 \,  (u_j^{(1)})^2-1) , \ldots,   
                                                \tfrac{\sigma^{(k)}_1}{2} a_k^2  (L_k \,  (u_j^{(k)})^2-1) \right).
\]
Let $\tilde{u}^{(i)}_j:=\sqrt{t_jL_i}\, a_i \,u^{(i)}_j$, $i=1, \ldots, k$,  $j=1, \ldots n_i$. Then,
\[
    (0, a_1^2, \ldots, a_k^2) = \pi^G  \left(
     \sum_{j=1}^n\sum_{i=1}^k \tau(\tilde{u}_j^{(i)}), \sum_{j=1}^n(\tilde{u}_j^{(1)})^2 , \ldots, \sum_{j=1}^n(\tilde{u}_j^{(k)})^2 \right).
\]
Set $\tilde{u}^{(i)}:=(\tilde{u}^{(i)}_1, \ldots, \tilde{u}^{(i)}_n): M \rightarrow \mathbb{R}^{n}$ for $i=1, \ldots, k$, and define 
\[
    v=(v^{(1)}_1, \ldots, v^{(1)}_{|G|n}, \ldots \ldots , v^{(k)}_1, \ldots , v^{(k)}_{|G|n}): 
    M \rightarrow \mathbb{R}^{|G|kn}
\]
by
\[
    v(x)=\frac{1}{|G|} \left( \tilde{u}^{(1)}(\sigma_1 \cdot x), \ldots, \tilde{u}^{(1)}(\sigma_{|G|} \cdot x), \ldots \ldots ,
            \tilde{u}^{(k)}(\sigma_1 \cdot x), \ldots, \tilde{u}^{(k)}(\sigma_{|G|} \cdot x) \right),
\]
where $\sigma_j$ is any enumeration of the elements of $G$. Then,
\begin{align*}
     0 & = \sum_{j=1}^{|G|n}  \sum_{i=1}^k \tau(v^{(i)}_j)  
     = \sum_{j=1}^{|G|n}  \sum_{i=1}^k (dv^{(i)}_j \otimes dv^{(i)}_j  -\tfrac{1}{2} \|\n v^{(i)}_j \|^2 g_0) \
          \quad \mbox{ on } M \\
     a_i^2& = \sum_{j=1}^{|G|n}  (v^{(i)}_j)^2    \;\, \mbox{ on } \Gamma_i, \; i=1, \ldots, k
\end{align*}
and $v^{(i)}_j \in V^{(i)}_1$, for each $i$, $1 \leq i \leq k$, so 
 \begin{align*}
       \nu v^{(i)}_j&=\sigma^{(i)}_1v^{(i)}_j   \;  \mbox{  on } \Gamma_i \\
        \nu v^{(i)}_j&=0  \;  \mbox{  on } \Gamma_l \mbox{ for } l \ne i      
 \end{align*}
for $j=1, \ldots, |G|n$. 
Thus 
\[
     v:=(v^{(1)}_1, \ldots, v^{(1)}_{|G|n}, \ldots \ldots, v^{(k)}_1, \ldots, v^{(k)}_{|G|n}): 
     M  \rightarrow B^{|G|n}(a_1) \times \cdots \times B^{|G|n}(a_k)
\]
is a conformal free boundary minimal immersion with $v(\Gamma_i)  \subset \partial B^{|G|n}(a_i)$. 
Furthermore, 
\[
      \left\|\nu  v  \right\|= \sigma^{(i)}_1 \left( \sum_{j=1}^{|G|n} v^{(i)}_j\right)^{\frac{1}{2}} =\sigma_1^{(i)} a_i
      \qquad \mbox{on } \Gamma_i,
\]
and since $v$ is conformal, $\sigma^{(i)}_1 a_i =\lambda$, where $\lambda$  is the conformal factor of $v$, and so $g_0$ is $F$-homothetic to the induced metric on  $M$ from $v$.    Finally, $v$ is equivariant by construction.

Suppose $v^{(i)}_1, \ldots , v^{(i)}_{|G|n}$ are linearly  dependent and span a space of dimension $n_i < |G|n$. Then  
$v(\Gamma_i) \subset B^{|G|n}(a_1) \times \cdots \times ( \partial B^{|G|n}(a_i) \cap \mathcal{P}) \times \cdots \times B^{|G|n}(a_k)$, where $\mathcal{P}$ is a linear subspace of $\mathbb{R}^{|G|n}$ of dimension $n_i$. This implies that 
$v(\Sigma) \subset B^{|G|n}(a_1) \times \cdots \times (B^{|G|n}(a_i) \cap \mathcal{P}) \times \cdots \times B^{|G|n}(a_k)$.
Setting $v^{(i)}=(v^{(i)}_1, \ldots , v^{(i)}_{|G|n}) : \Sigma \rightarrow \mathbb{R}^{|Gn}$, we may choose an orthogonal linear  transformation $T$ of  $\mathbb{R}^{|G|n}$ such that the first $n_i$ component functions of $T \circ v^{(i)}$  are linearly independent, and the remaining components are zero.
Applying this to each $i$, $1 \leq i \leq k$, we obtain linearly independent eigenfunctions $w_1^{(i)}, \ldots, w_{n_i}^{(i)} \in V_{g_0}^{(i)}$ such  that
$w:=(w^{(1)}_1, \ldots, w^{(1)}_{n_1}, \ldots, w^{(k)}_1, \ldots, w^{(k)}_{n_k}): M  
\rightarrow B^{n_1}(a_1) \times \cdots \times B^{n_k}(a_k)$
 is a conformal free boundary minimal immersion.  
\end{proof}

\begin{theorem} \label{theorem:characterization-conformal}
Let $G$ be a finite group acting on the surface $M$ by diffeomorphisms.
Suppose there exists a $G$-invariant metric $g_0$ on $M$ such that $F(g_0) = \sup_g F(g)$, where the supremum is over all smooth $G$-metrics on $M$  in the conformal class of $g_0$. Then there exist independent first eigenfunctions $u^{(i)}_1, \ldots, u^{(i)}_{n_i} \in V_1^{(i)}$ for $i=1, \ldots, k$ and an orthogonal representation $\rho: G \rightarrow O(n_1) \times \cdots \times O(n_k)$ such that 
\[
     u:=(u^{(1)}_1, \ldots, u^{(1)}_{n_1}, \ldots, u^{(k)}_1, \ldots, u^{(k)}_{n_k}): M 
     \rightarrow B^{n_1}(a_1) \times \cdots \times B^{n_k}(a_k)
\]
is a $\rho$-equivariant free boundary harmonic map with $u(\Gamma_i)\subset \partial B^{n_i}(a_i)$, where $B^{n_i}(a_i)$ is  the ball of radius $a_i$ centered at the origin in $\mathbb{R}^{n_i}$,  and the Hopf differential of $u$ is real on $\p M$.
\end{theorem}

The proof is analogous to \cite[Proposition 2.8]{FS3}.

\begin{proof}
Consider a smooth path $g(t)=\lambda(t)g_0$ of $G$-invariant metrics conformal to $g_0$, where $\lambda(t) \in C^\infty(M)$  are positive functions with $\lambda(0) \equiv 1$, so $g(0)=g_0$. Then $\sigma^{(i)}_1(t)$ is a Lipschitz function of $t$, and if $\dot{\sigma}_1^{(i)}(t)$ exists, 
\begin{align*}
       \dot{\sigma}_1^{(i)}(t) &= -\int_M \la \tau(u) , \dot{\lambda}(t) g_0 \ra \; dA_t
         -\tfrac{\sigma_1^{(i)}(t)}{2} \int_{\Gamma_i} u^2 \dot{\lambda}(t) g_0(T,T) \; ds_t  \\
         & =-\tfrac{\sigma_1^{(i)}(t)}{2} \int_{\Gamma_i} u^2 \dot{\lambda}(t) \; ds_0.
\end{align*}
As in the proof of Lemma \ref{lemma:indefinite}, for any $(f_1, \ldots f_k) \in L^2(\Gamma_1)^G \times \cdots \times L^2(\Gamma_k)^G$ there exist $u^{(i)} \in V^{(i)}_{g_0}$ with $\|u^{(i)}\|_{L^2(\Gamma_i)}=1$, $i=1, \ldots, k$, such that 
\[
   \la (f_1, \ldots, f_k), (
   \tfrac{\sigma^{(1)}_1}{2} a_1^2  (L_1\,  (u^{(1)})^2-1) , \ldots, \tfrac{\sigma^{(k)}_1}{2} a_k^2  (L_k \,  (u^{(k)})^2-1) )   
   \ra_{{L}^2}=0.
\]
Let $K$ be the convex hull in $L^2(\Gamma_1) \times \cdots \times L^2(\Gamma_k)$ of 
\[
        \{ \, ( \tfrac{\sigma^{(1)}_1}{2} a_1^2  (L_1\,  (u^{(1)})^2-1) , \ldots, \tfrac{\sigma^{(k)}_1}{2} a_k^2  (L_k \,  (u^{(k)})^2-1) ) \, : 
         u^{(i)} \in V^{(i)}_{g_0}, \, \|u^{(i)}\|_{L^2(\Gamma_i)}=1, \, i=1, \ldots, k\},
\]
and let $K^G=\pi^G(K)$, where $\pi^G: L^2(\Gamma_1) \times \cdots \times L^2(\Gamma_k) \rightarrow L^2(\Gamma_1)^G \times \cdots \times L^2(\Gamma_k)^G$ denotes the orthogonal projection.
Arguing as in the proof of Theorem \ref{theorem:characterization} we conclude that $(0, \ldots, 0) \in K^G$, and there exist eigenfunctions 
$u_1^{(i)}, \ldots, u_{n}^{(i)} \in V_{g_0}^{(i)}$, $i=1, \ldots, k$, such that
\[
    (a_1^2, \ldots, a_k^2) = \pi^G  \left(
     \sum_{j=1}^n({u}_j^{(1)})^2 , \ldots, \sum_{j=1}^n({u}_j^{(k)})^2 \right).
\]
Set ${u}^{(i)}:=({u}^{(i)}_1, \ldots, {u}^{(i)}_n): M \rightarrow \mathbb{R}^{n}$ for $i=1, \ldots, k$, and define 
\[
    v=(v^{(1)}_1, \ldots, v^{(1)}_{|G|n}, \ldots \ldots , v^{(k)}_1, \ldots , v^{(k)}_{|G|n}): 
    M \rightarrow \mathbb{R}^{|G|kn}
\]
by
\[
    v(x)=\frac{1}{|G|} \left( {u}^{(1)}(\sigma_1 \cdot x), \ldots, {u}^{(1)}(\sigma_{|G|} \cdot x), \ldots \ldots ,
            {u}^{(k)}(\sigma_1 \cdot x), \ldots, {u}^{(k)}(\sigma_{|G|} \cdot x) \right),
\]
where $\sigma_j$ is any enumeration of the elements of $G$. Then,
\[
     a_i^2 = \sum_{j=1}^{|G|n}  (v^{(i)}_j)^2    \;\, \mbox{ on } \Gamma_i, \; i=1, \ldots, k
\]
and $v^{(i)}_j \in V^{(i)}_1$, for each $i$, $1 \leq i \leq k$, so 
\begin{align*}
        \nu v^{(i)}_j &=\sigma^{(i)}_1v^{(i)}_j   \;  \mbox{  on } \Gamma_i \\
        \nu v^{(i)}_j &=0  \;  \mbox{  on } \Gamma_k \mbox{ for } k \ne i      
\end{align*}
for $j=1, \ldots, n$. 
Thus 
$v:=(v^{(1)}_1, \ldots, v^{(1)}_{|G|n}, \ldots, v^{(k)}_1, \ldots, v^{(k)}_{|G|n}): M  
\rightarrow B^{|G|n}(a_1) \times \cdots \times B^{|G|n}(a_k)$ is an equivariant  free boundary harmonic map with $u(\Gamma_i)\subset \partial B^{|G|n}(a_i)$. 
Furthermore,  on  $\Gamma_i$
\[
      \nu v= (0, \ldots, 0, \sigma^{(i)}_1 v^{(i)}_1, \ldots, \sigma^{(i)}_1 v^{(i)}_{|G|n}, 0, \ldots, 0)
\]
which is orthogonal to $v_T$ on $\Gamma_i$, and so
the Hopf differential of $v$ is real on the boundary of $M$.

As in the proof of \ref{theorem:characterization}, we may arrange  the linear independence of the eigenfunctions as in the  statement of the theorem.
\end{proof}

Let $\sigma^*(\gamma, b)$ denote the supremum of the  first normalized Steklov $\sigma_1L$ taken over all smooth metrics on the surface of genus $\gamma$ with $b\geq  1$ boundary components.
\begin{proposition} \label{proposition:SN-supremum}
Suppose $M$ has genus $\gamma$  and  let  $b_i$ denote the  number of  connected components of $\Gamma_i$.  Then
\[
      \sup_g \sigma_1^{(i)}(g) L_g(\Gamma_i) = \sigma^*(\gamma, b_i)
\]
where the supremum  is taken over all smooth metrics $g$ on $M$, and the supremum is not achieved if there are Neumann boundary components ($k >1$).
\end{proposition}
\begin{proof}
Let $g$ be a Riemannian metric on $M$.
First observe that the Steklov-Neumann eigenvalue $\sigma^{(i)}_1(g)$ depends only on the metric $g$ on the Steklov boundary $\Gamma_i$ and the conformal structure of the surface $(M,g)$. Any Riemann surface  is conformally equivalent to a closed Riemann surface $\bar{M}$ with finitely many disks removed. Let $M_i$ be the surface of genus $\gamma$ with $b_i$ boundary components gotten by removing the disks corresponding to the Steklov boundary components from the closed Riemann surface $\bar{M}$. We may choose a local complex coordinate $z$ on any of the Neumann disks so that the disk is the unit disk $|z| \leq 1$. Given $0 < r < 1$, let $M_i^r$ be the surface gotten by removing the disk $|z| <r$ in each Neumann disk from the surface $M_i$. Then $M \subset M_i^r \subset M_i$, and 
\begin{equation} \label{equation:monotone} 
     \sigma^{(i)}_1(g) = \inf_{\substack{u \in W^{1,2}(M) \\ \int_{\Gamma_i} u=0}}
       \frac{\int_M \|\nabla u\|^2 \,da}{\int_{\Gamma_i} u^2 \,ds} 
     \leq  \inf_{\substack{u \in W^{1,2}(M_i^r) \\ \int_{\Gamma_i} u=0}}
       \frac{\int_{M_i^r} \|\nabla u\|^2 \,da}{\int_{\Gamma_i} u^2 \,ds}  
     \leq \inf_{\substack{u \in W^{1,2}(M_i) \\ \int_{\Gamma_i} u=0}}
       \frac{\int_{M_i} \|\nabla u\|^2 \,da}{\int_{\Gamma_i} u^2 \,ds}   
\end{equation}
and it follows that $\sigma_1^{(i)}(g) L_g(\Gamma_i) \leq \sigma^*(\gamma, b_i)$. 

Moreover, if $u_r$ is a first Steklov-Neumann eigenfunction of $M_i^r$, then there exists a sequence $r_l \rightarrow 0$ such that $u_{r_l}$ converges as $l \rightarrow \infty$ in $C^2$ on compact subsets of $M_i$ away from the centers of the Neumann disks to a first Steklov eigenfunction $u$ of $M_i$, and $\sigma_1(M_i) = \lim_{r \rightarrow 0} \sigma_1^{(i)} (M_i^r)$ (see for example the proof of \cite[Theorem 1.2]{FS4} for further details). In particular, if we choose a Riemannian surface of genus $\gamma$ with $b_i$ boundary components with first normalized Steklov eigenvalue $\sigma_1L$ close to $\sigma^*(\gamma,b_i)$, then by removing $\sum_{j\neq i} b_j$ sufficiently small disjoint disks from the surface, we may obtain a surface with first normalized Steklov-Neumann eigenvalue $\sigma^{(i)}_1L$ arbitrarily close to $\sigma^*(\gamma,b_i)$. Therefore, $\sup_g \sigma_1^{(i)}(g) L_g(\Gamma_i) = \sigma^*(\gamma, b_i)$.

Suppose $k>1$ and there existed a metric $g$ that achieved the supremum of $\sigma_i^{(i)}L$ over all smooth metrics on $M$. Since $k>1$, the first inequality in (\ref{equation:monotone}) is strict for $0<r<1$, and we have
$\sigma_1^{(i)}(g) L_g(\Gamma_i) < \sigma_1^{(i)}(M_i^r) L_g(\Gamma_i)$, a contradiction. Therefore the supremum is not achieved.
\end{proof}
\begin{proposition} \label{proposition:F-supremum}
Suppose $M$ has genus $0$  and  let  $b_i$ denote the  number of  connected components of $\Gamma_i$.  Then
\[
      \sup_g F(g) =   \sum_{i=1}^k a_i^2 \, \sigma^* (0,b_i)
\]
where the supremum  is taken over all smooth metrics $g$ on $M$, and the supremum is not achieved if $k >1$.
\end{proposition}
\begin{proof}
By Proposition \ref{proposition:SN-supremum} we have $\sup_g F(g) \leq   \sum_{i=1}^k a_i^2 \, \sigma^* (0,b_i)$.
Let $M_{i}$ be a Riemannian surface of genus $0$ with $b_j$ boundary components $\Gamma_i$ such that $\sigma_{1}(M_i)L(\Gamma_i)$ is arbitrarily close to $\sigma^*_{1}(0,b_i)$. We may glue the surfaces $M_{1}, \ldots, M_{k}$ together using cylindrical necks between interior points to obtain a Riemannian surface $M$ of genus $0$ with $b_1 + \cdots + b_k$ boundary components, and such that $\sigma^{(i)}_1(M)L(\Gamma_i)$ (the first normalized eigenvalue of the Steklov-Neumann problem which is Steklov on $\Gamma_i$ and Neumann on the other boundary components)  is  arbitrarily close to $\sigma^*(0,b_i)$, and  therefore  $F(M)$ is arbitrarily close to $\sum_{i=1}^k a_i^2 \, \sigma^* (0,b_i)$ (the details are as in  the the proof of \cite[Theorem 1.2]{FS4}).

Suppose there existed a metric $g$ on $M$ that achieved the supremum. Then 
$\sigma_1^{(i)}(g)L_g(\Gamma_i) =\sigma^* (0,b_i)$, a contradiction since by  Proposition \ref{proposition:SN-supremum} the supremum  is not achieved.
\end{proof}
For example, consider a surface of genus zero with six boundary components divided into three pairs of boundary components $\Gamma_1, \, \Gamma_2, \, \Gamma_3$. By Proposition \ref{proposition:F-supremum} there is no metric on $M$  that achieves the supremum of $a_1^2\sigma_1^{(1)}(g)L_g(\Gamma_1)+a_2^2\sigma_1^{(2)}(g)L_g(\Gamma_2)+a_3^2\sigma_1^{(3)}(g)L_g(\Gamma_3)$ over all smooth metrics on $M$, but the supremum is achieved in the limit by a sequence of metrics degenerating to the disjoint union of three identical copies of the critical catenoid.

Using similar arguments to the proof of Proposition \ref{proposition:F-supremum} we have the following.
\begin{proposition}
Suppose $M$ has genus $\gamma$  and  let  $b_i$ denote the  number of  connected components of $\Gamma_i$.  Then
\[
      \sup_g F(g) \geq\max_{\substack{\gamma_1+ \cdots + \gamma_k=\gamma \\ 
                     \gamma_j \geq 0 \; \forall \,j 
                    }}
      \; \sum_{i=1}^k a_i^2 \, \sigma^* (\gamma_i,b_i)
\]
where the supremum  is taken over all smooth metrics $g$ on $M$. 
\end{proposition}

\section{The canonical model and the moduli space} \label{section:canonical-model}

We now restrict ourselves to a genus $0$ Riemann surface $M$ with $6$ boundary components. By a classical theorem of
Koebe $M$ is conformally diffeomorphic to an open subset of $\mathbb S^2$ with circular boundary components. We have seen that in order
to maximize our eigenvalue problem we must impose enough symmetry. We have the following lemma.
\begin{lemma} Any conformal automorphism $\rho$ of $M$ (orientation preserving or reversing) is the restriction of a conformal automorphism of
$\mathbb S^2$. Furthermore if $\rho$ preserves each boundary component of $M$ then $\rho$ is the identity.
\end{lemma}
\begin{proof} The first statement is well known and can be found in \cite{Ko}. Now if we assume that $\rho$ preserves each boundary component,
then we can apply a M\"obius transformation to place two of the boundary components $\gamma_1,\gamma_2$ symmetrically about the 
$xy$-plane and thus $\rho$ preserves the annulus bounded by $\gamma_1$ and $\gamma_2$. It follows that $\rho$ must be a rotation about
the $z$-axis. Thus $\rho$ cannot preserve any of the other boundary circles unless $\rho$ is the identity.
\end{proof}
We now assume that the $6$ boundary components are divided into $3$ groups $\Gamma_1,\Gamma_2,\Gamma_3$ with each group containing
a pair of boundary components; $\Gamma_j=\{\gamma^{(j)}_1,\gamma^{(j)}_2\}$ for $j=1,2,3$. We then assume that we have an orientation 
reversing conformal involution $\rho_j$ that interchanges the components in $\Gamma_j$ and preserves the other $4$ boundary components. Thus
we have $\rho_j^2=I$, the identity. Furthermore since the commutator $\rho_j\rho_k\rho_j\rho_k$ preserves all of the boundary components
it is the identity by the lemma, and it follows that $\rho_j$ commutes with $\rho_k$. Thus the conformal symmetry group of $M$ is the group 
$G=\mathbb Z_2\times \mathbb Z_2\times \mathbb Z_2$. 

By the previous discussion we may assume that $M$ is a subdomain in $\mathbb S^2$ with circular boundary components and that $G$ is
a subgroup of the M\"obius group of conformal transformations of $\mathbb S^2$. Since $G$ is compact it follows that $G$ is conjugate to a
subgroup of the orthogonal group. Thus after applying a conformal transformation we may assume that the elements of $G$ lie in the orthogonal
group. Now any orthogonal matrix $\rho$ with $\rho^2=I$ is conjugate in the orthogonal group to a diagonal matrix with $1$ and $-1$ on
the diagonal. Since the group $G$ is abelian it follows that there is a Euclidean coordinate system so that $G$ consists of the group of diagonal
matrices with $1$ and $-1$ on the diagonal. Thus the domain $M$ is bounded by circles centered on the coordinate axes and is invariant
under reflections in the coordinate planes. We refer to such a representation of $M$ as the {\it canonical model} and we assume that $\Gamma_1$
consists of the pair of boundary circles centered on the $x$-axis and $\rho_1$ is reflection in the $yz$-plane. Similarly $\Gamma_2$ consists of
the boundary circles centered on the $y$-axis and $\rho_2$ is reflection in the $xz$-plane. Finally $\Gamma_3$ consists of the boundary
circles centered on the $z$-axis and $\rho_3$ is reflection in the $xy$-plane. We may parametrize the moduli space of such canonical models by
the radii of the circles $r_1,r_2,r_3$ where the radii are restricted by the condition that the circles are disjoint in $\mathbb S^2$.

To summarize we have explicitly identified that moduli space $\mathcal M$ of the surfaces we are considering as a convex subset of $\mathbb R^3$
\[ \mathcal M=\{(r_1,r_2,r_3):\ r_i>0,\ r_i+r_j<\pi/2\ \forall{i,j}, i\ne j\}.
\]
Note that the inequality $r_i+r_j<\pi/2$ is the condition that the circles in $\Gamma_i$ are disjoint and bound disjoint disks in $\mathbb S^2$.

\section{Symmetries of eigenfunctions}

As in section \ref{section:canonical-model}, we  assume that $M$ is a Riemannian surface of genus 0 with 6 boundary components $\Gamma_1,\Gamma_2,\Gamma_3$ with each group containing a pair of boundary components; $\Gamma_i=\{\gamma^{(i)}_1,\gamma^{(i)}_2\}$ for $i=1,2,3$, and that there is an orientation 
reversing conformal involution $\rho_i$ that interchanges the components in $\Gamma_i$ and preserves the other $4$ boundary components.
Fix $i$ and let  $V^{(i)}$  denote the eigenspace corresponding to the first nonzero eigenvalue of $(SN_i)$. Then $V$ admits a direct sum decomposition,
\[
     V^{(i)}=V^{(i)}_{eee} \oplus V^{(i)}_{eeo} \oplus V^{(i)}_{eoe} \oplus V^{(i)}_{eoo} 
                \oplus V^{(i)}_{oee} \oplus V^{(i)}_{oeo} \oplus V^{(i)}_{ooe} \oplus V^{(i)}_{ooo} 
\]
where the three subscripts (e or o) denote that the eigenfunctions are even or odd under $\rho_1$,  $\rho_2$,  $\rho_3$ respectively.
\begin{proposition}  \label{proposition:even-odd}
For each $i=1,2,3$,  $V^{(i)}=V^{(i)}_{eeo} \oplus V^{(i)}_{eoe} \oplus V^{(i)}_{oee}$, and each summand has dimension at most 1.
\end{proposition}
\begin{proof}
We divide the proof into  four parts.

\vspace{2mm}

\noindent
1) $V^{(i)}_{eoo}=V^{(i)}_{oeo}=V^{(i)}_{ooe}=V^{(i)}_{ooo}=0$:

\vspace{1mm}

\noindent
Let $u$ be a first eigenfunction of $(SN_i)$. If $u$ is odd under $\rho_j$ and $\rho_k$ for $j \ne k$, then $u$ has four zeros on  $\Gamma_l$, $l \ne j, k$, at the fixed points of $\rho_j$  and $\rho_k$. But a first Steklov-Neumann eigenfunction can have at most two zeros on a boundary component (see for example \cite[Theorem 2.3]{FS2}).

\vspace{2mm}

\noindent
2) $V^{(i)}_{eee}=0$:

\vspace{1mm}

\noindent
Without loss of generality suppose $i=1$, and let $u$ be a first eigenfunction of $(SN^{(1)})$.  Then $\int_{\Gamma_1}  u \, ds=0$, and if $u$ is even  under $\rho_1$ then
$\int_{\gamma^{(1)}_1} u \, ds=0$ and $\int_{\gamma^{(1)}_2} u \, ds=0$. This implies  that $u$ has two simple zeros on  $\gamma^{(1)}_1$. If $u$ is also even under $\rho_2$ and $\rho_3$, then the zeros can't be on the fixed point set of  $\rho_2$ or $\rho_3$ since $u$  changes sign  across  the  zeros since they are   simple. Therefore,  $u$  has at least 4 zeros, a contradiction.

\vspace{2mm}

\noindent
In the next two parts suppose $i=1$. The cases  $i=2, 3$ are analogous.

\vspace{2mm}

\noindent
3) $\dim V^{(1)}_{eoe} \leq  1$  and  $\dim V^{(1)}_{eeo} \leq  1$:

\vspace{1mm}

\noindent
Choose a point  $p \in \gamma^{(1)}_1$  in the  fixed point set of  $\rho_2$.  Consider the linear transformation
$T: V_{eeo}^{(1)} \rightarrow \mathbb{R}$ given  by $T(u)=u(p)$. If $\dim V^{(1)}_{eeo} >1$ then  $\ker(T) \neq \{0\}$ and there  exists $u \in V^{(1)}_{eeo}$ such that $u(p)=0$. Since $\int_{\Gamma_1}  u \, ds=0$ and $u$ is even  under $\rho_1$,
$\int_{\gamma^{(1)}_1} u \, ds=0$ and $u$ has two simple zeros on  $\gamma^{(1)}_1$. Since a first Steklov-Neumann eigenfunction can have at most two zeros on a boundary component,  $p$ must  be one of these simple zeros. Since $p$ is simple,  $u$  changes sign across $p$,  a  contradiction since $p$ is on the fixed point set of  $\rho_2$ and  $u$  is even under $\rho_2$.  Therefore,  $\dim V^{(1)}_{eeo} \leq  1$.  Similarly, $\dim V^{(1)}_{eoe} \leq  1$.

\vspace{2mm}

\noindent
4) $\dim V^{(1)}_{oee} \leq  1$:

\vspace{1mm}

\noindent
Consider  $T: V^{(1)}_{oee}  \rightarrow \mathbb{R}$ given by $T(u) = \int_{\gamma^{(1)}_1} u  \,  ds$. If $\dim V^{(1)}_{oee} >  1$ then  there exists $u \in V^{(1)}_{oee}$ such that $\int_{\gamma^{(1)}_1} u \, ds=0$,  which implies that  $u$ has two simple zeros on  $\gamma^{(1)}_1$. Since $u$  is even under $\rho_2$  and $\rho_3$ the zeros can't be on the fixed point sets of $\rho_2$ or $\rho_3$, and  so there are  four  zeros, a  contradiction.  Therefore, $\dim V^{(1)}_{oee} \leq  1$.
\end{proof}

We now show that for a maximizing metric, the linearly independent eigenfunctions in Theorem  \ref{theorem:characterization-conformal} that give a free boundary harmonic map from $M$  into a product of balls can be chosen to lie in $V^{(i)}_{eoe}, V^{(i)}_{oee}, V^{(i)}_{eeo}$.
\begin{proposition}
Suppose there exists a metric $g_0$ on $M$, invariant under  $\rho_1, \rho_2, \rho_3$, such that $F(g_0) = \sup_g F(g)$, where the supremum is over all smooth metrics on $M$  in the conformal class of $g_0$ that  are invariant  under $\rho_1, \rho_2,  \rho_3$. Then there exist functions $\hat{u}^{(i)} \in V^{(i)}_{eoe}$, $\hat{v}^{(i)} \in V^{(i)}_{oee}$, $\hat{w}^{(i)} \in V_{eeo}^{(i)}$ for $i=1, 2, 3$ such that 
\[
     \varphi:=(\hat{u}^{(1)}, \hat{v}^{(1)}, \hat{w}^{(1)}, 
     \hat{u}^{(2)}, \hat{v}^{(2)}, \hat{w}^{(2)}, \hat{u}^{(3)}, \hat{v}^{(3)}, \hat{w}^{(3)}): M 
     \rightarrow B^{3}(a_1) \times B^{3}(a_2) \times B^{3}(a_3)
\]
is a  free boundary harmonic map with $\varphi(\Gamma_i)\subset \partial B^{3}(a_i)$  and the Hopf differential of $\varphi$ is real on $\p M$.
\end{proposition}
\begin{proof}
By Proposition  \ref{proposition:even-odd},  $\dim V^{(i)} \leq 3$,  and  so by Theorem \ref{theorem:characterization-conformal} we only need to show that we can choose the component  functions of $\varphi$ to lie in the $eoe$, $oee$, $eeo$ eigenspaces. Fix  $i$, $1 \leq i  \leq 3$. By Theorem \ref{theorem:characterization-conformal}  there exist functions  $u, v, w  \in  V^{(i)}$ such that  $u^2+v^2+w^2=a_i^2$.  We  may decompose each of these functions uniquely into a sum of a  function that is even under  $\rho_1$ and a  function that is odd  under $\rho_1$:
\[
     (u_e  +u_o)^2 + (v_e+v_o)^2+(w_e+w_o)^2=a_i^2.
\]
Composing this with $\rho_1$, we obtain
\[
     (u_e  -u_o)^2 + (v_e-v_o)^2+(w_e-w_o)^2=a_i^2,
\]
which implies that
\[
    u_e^2+u_o^2+v_e^2+v_o^2+w_e^2+w_o^2=a_i^2.
\]
Since $u_o$, $v_o$, $w_o$ are odd under $\rho_1$, by Proposition \ref{proposition:even-odd} they must even under $\rho_2$ and $\rho_3$, and since $\dim  V^{(i)}_{oee}  \leq 1$, each must be a multiple of some function, and we have 
\[
    u_e^2+v_e^2+w_e^2=a_i^2 -\hat{w}^2
\]
for some function $\hat{w} \in V^{(i)}_{oee}$.
We now decompose each of $u_e$,  $v_e$, $w_e$ into a sum of a function that is even under  $\rho_2$ and a function that is odd  under $\rho_2$:
\[
   (u_{ee}+u_{eo})^2  + (v_{ee}+v_{eo})^2 +(w_{ee}+w_{eo})^2=a_i^2-\hat{w}^2.
\]
Composing with  $\rho_2$, we obtain
\[
   (u_{ee}-u_{eo})^2  + (v_{ee}-v_{eo})^2 +(w_{ee}-w_{eo})^2=a_i^2-\hat{w}^2
\]
which implies that
\[
   u_{ee}^2+u_{eo}^2  + v_{ee}^2+v_{eo}^2 +w_{ee}^2+w_{eo}^2=a_i^2-\hat{w}^2.
\]
Since $u_{ee}$, $v_{ee}$, $w_{ee}$ are even under $\rho_1$  and $\rho_2$, by Proposition \ref{proposition:even-odd} they must be odd under $\rho_3$, and since $\dim  V^{(i)}_{eeo}  \leq 1$, each must be a multiple of some function, and we have 
\[
   u_{eo}^2  +v_{eo}^2 +w_{eo}^2=a_i^2-\hat{w}^2 -\hat{v}^2
\]
for some function $\hat{v} \in V^{(i)}_{eeo}$.
We now decompose each of $u_{eo}$,  $v_{eo}$, $w_{eo}$ into a sum of a function that is even under $\rho_3$ and a function that is odd  under $\rho_3$, and since $V^{(i)}_{eoo}=0$ (by Proposition \ref{proposition:even-odd}) we obtain
\[
   u_{eoe}^2  +v_{eoe}^2 +w_{eoe}^2=a_i^2-\hat{w}^2 -\hat{v}^2.
\]
Since $\dim V^{(i)}_{eoe} \leq 1$, $u_{eoe}^2  +v_{eoe}^2 +w_{eoe}^2=\hat{u}^2$  for some function $\hat{u} \in V^{(i)}_{eoe}$, and we obtain
\[
     \hat{u}^2 +  \hat{v}^2 + \hat{w}^2  =  a_i^2.
\]
Doing this for each $i$, $1 \leq i \leq 3$, we obtain the result.
\end{proof}

Recall that the Steklov-Neumann eigenvalue $\sigma^{(i)}_1(g)$ depends only on the metric $g$ on the Steklov boundary $\Gamma_i$ and the conformal structure of the surface $(M,g)$. As discussed in section \ref{section:canonical-model}, $M$ is conformally  equivalent  to the {\em canonical model}, 
a subdomain in $\mathbb S^2$ with circular boundary components of radii $r_1, r_2, r_3$ centered at points 
on the coordinate axes, that is invariant under reflections in the coordinate planes. 
 
Fix $r_2, r_3$ and let $\sigma_1^{(i)}(r_1)$ denote the first Steklov-Neumann eigenvalue of  $(SN_i)$ of the canonical model $M_{r_1}$  with the spherical metric as a function of $r_1$. 
Let $\sigma_1^{(i)}(M_{r_1},g)$ denote the eigenvalue with respect to a conformal metric $g$ on $M_{r_1}$ 
that is invariant under the reflections $\rho_j$ for $j=1,2,3$.
\begin{lemma} \label{lemma:epsilon23}
For $\epsilon >0$ small, $\sigma_1^{(i)}(M_\epsilon,g) \geq \sigma_1^{(i)}(M_0,g) - c \epsilon^2$, for some constant $c$ independent of $\epsilon$,  for $i=2, 3$.
\end{lemma}
\begin{proof}
For $i=1$ or $2$, let $u$ be a first eigenfunction for $\sigma^{(i)}_1(M_\epsilon, g)$.
Since a  disk
centered at $p_1$ in $\mathbb{S}^2$ is conformally equivalent to a Euclidean disk,  without loss of generality we will work on the Euclidean unit disk $D$. We may write $\varphi=u|_{\partial D}$ 
as
\[
    \varphi (\theta)=a_0 + \sum_{n=1}^\infty a_n \cos(n\theta) + b_n \sin(n \theta).
\]
Since $u$ is a Steklov-Neumann eigenfunction, on $D \setminus D_\epsilon$, $u$ satisfies
\begin{align*}
     \Delta u &= 0 \mbox{ on } D \setminus D_\epsilon \\
     u &= \varphi \mbox{ on } \partial D \\
     \fder{u}{r} &=0 \mbox{ on } \partial D_\epsilon,
\end{align*}
which implies that 
\[
    u= a_0 + \sum_{n=1}^\infty a_n \frac{r^n + \epsilon^{2n} r^{-n}}{1+\epsilon^{2n}} \cos(n\theta)
                + \sum_{n=1}^\infty a_n \frac{r^n + \epsilon^{2n} r^{-n}}{1+\epsilon^{2n}} \sin(n\theta).
\]
On $\partial D_\epsilon$ we have
\[
  \fder{u}{\theta}(\epsilon, \theta)
  = \sum_{i=1}^\infty \frac{2\epsilon^n}{1+\epsilon^{2n}} \left(nb_n \cos(n\theta) -na_n \sin(n\theta)\right).
\]
Since $\varphi$ is smooth, $|na_n| \leq c$ and $|nb_n| \leq c$ for a constant $c$ independent of $n$, and
\[
    \left|\fder{u}{\theta}(\epsilon,\theta)\right|
    \leq \sum_{n=1}^\infty c \,\epsilon^n 
    =c \, \frac{\epsilon}{1-\epsilon}.
\]
This implies, after subtracting away a constant term, that $|u(\epsilon, \theta)|\leq c \, \epsilon^2$.
Extend $u$ to $D_\epsilon$ by
\[
     u(r,\theta)=\frac{r}{\epsilon}u(\epsilon, \theta). 
\]
Then
\[
     \|\nabla u\|^2(r,\theta) 
     = \frac{1}{\epsilon^2} |u(\epsilon,\theta)|^2 
         + \frac{1}{r^2} \frac{r^2}{\epsilon^2} \left|\fder{u}{\theta}(\epsilon,\theta)\right|^2 \leq C 
     \quad \mbox{ on } D_\epsilon,
\]
where $C$ is independent of $\epsilon$. Therefore,
$
     \int_{D_\epsilon} |\nabla u|^2 \,da \leq c \epsilon^2
$
and we have
\begin{align*}
     \sigma^{(i)}_1(M_0,g) 
      & \leq \frac{ \int_{M_0} \|\nabla u\|^2 \, da}{\int_{\Gamma_i} u^2 \, ds} 
      = \frac{ \int_{M_\epsilon} \|\nabla u\|^2 \, da}{\int_{\Gamma_i} u^2 \, ds}
          + \frac{ \int_{D_\epsilon} \|\nabla u\|^2 \, da}{\int_{\Gamma_i} u^2 \, ds} \\
      &= \sigma^{(i)}_1(M_\epsilon,g) + \frac{ \int_{D_\epsilon} \|\nabla u\|^2 \, da}{\int_{\Gamma_i} u^2 \, ds} \\
      & \leq \sigma^{(i)}_1(M_\epsilon,g) + c\epsilon^2.
\end{align*}
\end{proof}

We now derive an upper bound on $\epsilon\sigma_1^{(1)}(\epsilon)$ which (we will show in the next section) is up to a constant the largest 
normalized first eigenvalue of $(SN_1)$ for $M_\epsilon$ over conformal metrics on $M_\epsilon$. 
\begin{lemma} \label{lemma:upper-lower-bound}
Assume that $r_2$ and $r_3$ are fixed and that $\epsilon$ is small. We then have the upper bound 
\[  \epsilon\sigma_1^{(1)}(\epsilon)\leq \frac{c}{|\log(\epsilon)|}
\]
for a positive constant $c$ independent of $\epsilon$.
\end{lemma}

\begin{proof} We assume that $r_0>0$ is such that the disk of radius $r_0$
concentric with $\gamma_1^{(1)}$ is disjoint from $\Gamma_2$ and $\Gamma_3$. We choose a test function $\phi$ which is odd under
$\rho_1$ and which is zero outside $D_{r_0}$ and its image under $\rho_1$. Inside $D_{r_0}\setminus D_\epsilon$ we choose
\[ \phi(r)=1-\frac{\log(r/\epsilon)}{\log(r_0/\epsilon)},\ \ \epsilon\leq r\leq r_0
\] 
where $r$ is the spherical distance from the center of $D_{r_0}$. A straightforward calculation shows that 
\[ \int_{M_\epsilon}\|\nabla\phi\|^2\ da\leq \frac{c}{|\log(\epsilon)|}.
\]
This in turn implies 
\[ \epsilon\sigma_1^{(1)}(\epsilon)\leq \epsilon \frac{\int_{M_\epsilon}\|\nabla\phi\|^2\ da}{\int_{\partial M_\epsilon}\phi^2\ ds}\leq \frac{c_2}{|\log(\epsilon)|}
\]
for a constant $c_2$. This is the upper bound.
\end{proof}

We see that the above argument does not require the metric on $M_\epsilon$ to be spherical. Let $g$ be any conformal metric on $M_\epsilon$,
which is invariant under the reflections $\rho_i$ for $i=1,2,3$ and let $L(\Gamma_1)$ denote the total length of the boundary curves in $\Gamma_1$.
\begin{proposition}  \label{proposition:upper-bound}
For $\epsilon$ small we have 
\[ L(\Gamma_1)\sigma_1^{(1)}(M_\epsilon,g)\leq \frac{c_2}{|\log(\epsilon)|}.
\]
\end{proposition}

We will prove a corresponding lower bound for a maximizing metric in the conformal class in the next section.

\section{Odd eigenvalues and minimal surfaces in $\mathbb R^3$}

In this section we consider the problem of finding critical points of $F$ among surfaces with the symmetries we are imposing and also with the restriction
that our eigenfunctions be odd under the reflection which interchanges the Steklov boundary components. This will lead us to a variational construction
of the Schwarz $P$-surface and generalizations of it to rectangular solids with arbitrary side lengths. Recall that the moduli space $\mathcal M$ of our surfaces can be described using the canonical model in terms of three positive radii $r_1,r_2,r_3$ with the restriction that $r_i+r_j<\pi/2$ for $i\neq j$,
\[ \mathcal M=\{(r_1,r_2,r_3):\ r_i>0,\ r_j+r_k<\tfrac{\pi}{2}\ \forall \, i,j,k\ \mbox{with}\ j\neq k\}.
\]
Note that the restriction $r_j+r_k<\tfrac{\pi}{2}$ guarantees that the circles are disjoint in $\mathbb S^2$. We see from its definition that $\mathcal M$ is an open
convex subset of $\mathbb R^3$.

Now given a point $M\in \mathcal M$ and a conformal metric $g$ on $M$ we let $\sigma_1^{(i)}(g)$ denote the lowest eigenvalue for the 
problem $(SN_i)$. Note that the eigenvalue depends only on the metric $g$ restricted to the components of $\Gamma_i$ and the conformal structure.
We also let $\sigma_1^{(ie)}(g)$ (resp. $\sigma_1^{(io)}(g)$) denote the lowest eigenvalues when the competing functions are restricted to functions which are even (resp. odd) under the reflection $\rho_i$. We see from the discussion of the previous section that for any metric $g$
\begin{equation} \label{equation:min-even-odd}
      \sigma_1^{(i)}(g)=\min\{\sigma_1^{(ie)}(g),\sigma_1^{(io)}(g)\}.
\end{equation}

In this section we will discuss geometry related to the choice of odd eigenfunctions,  and  consider the extremal problem for the functional
\[
    F^o(g)=\sum_{i=1}^3 a_i^2L_g(\Gamma_i)\sigma_1^{(io)}(g).
\]
 We recall from Proposition \ref{proposition:even-odd} that the multiplicity of 
$\sigma_1^{(io)}(g)$ is one for any metric $g$ (with the reflection symmetries). For any $M\in \mathcal M$ we will explicitly describe the metric and
eigenfunction which maximizes $\sigma_1^{(io)}(g)$ over conformal metrics $g$ with fixed length $L_i$ of $\Gamma_i$. 
Let $u_i$ be the harmonic function on $M$ that satisfies boundary conditions $u_i=1$ on $\gamma_1^{(i)}$, 
$u_i=-1$ on $\gamma_2^{(i)}$, and $\nu u_i=0$ on each component of $\Gamma_j$ for $j\neq i$ where $\nu$ is the unit normal for any conformal
metric near $\Gamma_j$ (the Neumann condition is metric independent). We now let $\lambda_i=\nu u_i$ on $\gamma_1^{(i)}$ and 
$\lambda_i=-\nu u_i$ on $\gamma_2^{(i)}$ where $\nu$ denotes the outward unit normal with respect the standard metric on $\mathbb S^2$. We then
note that the metric $\lambda_i$ is symmetric under $\rho_i$ since both $u_i$ and $\nu$ are odd under $\rho_i$. Thus $\lambda_i$ extends to
a conformal metric on $M$, and we now show that this metric has an extremal property. We let $L_i$ denote the length of $\Gamma_i$ with respect
to $\lambda_i$. Note also that by definition we see that $u_i$ is an eigenfunction for $(SN_i)$ with eigenvalue $1$. 

We explicitly state the following result which is important both for the maximization of $F$ and $F^o$.
\begin{lemma}  \label{lemma:maximization}
 If we fix a conformal structure on $M$, then a metric $g$ is determined by even density functions $\lambda_i$ on $\Gamma_i$ for 
$i=1,2,3$. A metric $g$ maximizes $F$ (resp. $F^o$) in the conformal class if and only if $\lambda_i$ maximizes $(SN_i)$ (resp. $(SN_i^o)$)
for each $i$.
\end{lemma}
\begin{proof} This follows from the previous observations that the eigenvalues of both $(SN_i)$ and $(SN_i^o)$ depend only on the metric on 
$\Gamma_i$, and so a maximizer of either $F$ or $F^o$ is obtained by a metric whose value on $\Gamma_i$ maximizes the corresponding
eigenvalue $(SN_i)$ or $(SN_i^o)$.
\end{proof}

We now describe the maximizer in a conformal class of $F^o$.

\begin{proposition} For any conformal metric $g$ on $M$ we have $L_g(\Gamma_i)\sigma_1^{(i)}(g)\leq L_g(\Gamma_i)\sigma_1^{(io)}(g)\leq L_i$ where equality holds in the first inequality only if the first eigenfunction  for $(SN_i)$  with  metric $g$ is odd. Although $u_i$ might not be a first eigenfunction, it is always a first {\it odd} eigenfunction in that it minimizes the Raleigh quotient among functions that are odd under $\rho_i$. Thus we have
\[ L_i=\max\{L_g(\Gamma_i)\sigma_1^{(io)}(g):\ g\ \mbox{conformal metric}\}.
\]
\end{proposition}

\begin{proof} Suppose $g$ is any symmetric conformal metric on $M$ and that its restriction to each component of $\Gamma_i$ is given by a positive
function $\mu_i$ which is invariant under each reflection $\rho_j$. Since the function $u_i$ is odd under $\rho_i$ it follows that 
$\int_{\Gamma_i}u_i\mu_i\ ds=0$. Therefore by the variational characterization of $\sigma_1^{(io)}$ we have
\[ \sigma_1^{(io)}(g)\int_{\Gamma_i}u_i^2\ \mu_i\ ds\leq \int_M\|\nabla u_i\|^2\ da=\int_{\Gamma_i}u_i\ \nu u_i\ ds.
\]
Since $u_i^2=1$ on $\Gamma_i$ and the definition of $\lambda_i$ this implies
\[ L_g(\Gamma_i)\sigma_1^{(io)}(g)\leq L_i
\]
which implies the second inequality. The first inequality follows from (\ref{equation:min-even-odd}).

Finally we show that $u_i$ is the lowest odd eigenfunction of $(SN_i)$ with metric $\lambda_i$. If we let $\mathbb S^2_+$ denote the hemisphere
which contains $\gamma_1^{(i)}$ we observe that the Raleigh quotient for odd functions can be written
\[ Q(v)=\frac{\int_{M\cap \mathbb S^2_+}\|\nabla v\|^2\ da}{\int_{\gamma_1^{(i)}} v^2\ \lambda_i\ ds}
\]
where $v$ is a function on $M\cap\mathbb S^2_+$ with $v=0$ on $M\cap\partial \mathbb S^2_+$. Such a function can then be extended to $M$ by odd
reflection and the extension has the above Rayleigh quotient noting that the oddness implies the boundary integral over $\Gamma_i$ is zero.

Since $Q(|v|)=Q(v)$, we see that any lowest odd eigenfunction must be of one sign on $M\cap\mathbb S^2_+$. Thus if the lowest eigenvalue $\sigma$
were less than $1$ we would have a positive eigenfunction $v_i$ on $M\cap\mathbb S^2_+$. We would then have
\[  \int_{\gamma_1^{(i)}}u_iv_i\ \lambda_i\ ds=\int_{M\cap\mathbb S^2_+}\langle\nabla u_i,\nabla v_i\rangle\ da=
\sigma \int_{\gamma_1^{(i)}}u_iv_i\ \lambda_i\ ds.
\]
If $\sigma<1$ this would imply $\int_{\gamma_1^{(i)}}u_iv_i\ \lambda_i\ ds=0$ contradicting the fact that both functions are positive on 
$\gamma_1^{(i)}$. Therefore $u_i$ is a least odd eigenfunction for $(SN_i)$ with metric $\lambda_i$.
\end{proof}

We now return to the bounds of the previous section in case $M_\epsilon$ is the surface with $r_1=\epsilon$ and $r_2, r_3$ fixed. In general, we let $M_p$ denote the surface with $(r_1,r_2,r_3)=p$. The companion
lower bound of Proposition \ref{proposition:upper-bound} is the following.

\begin{proposition}\label{proposition:lower-bound} Suppose $\{i,j,k\}=\{1,2,3\}$. There is a positive constant $c$ depending on $r_j$ and $r_k$ so that if $M_\epsilon$ denotes the surface corresponding to $r_i=\epsilon$ and $r_j,r_k$ we have for $\epsilon$ small
\[ \frac{c}{|\log(\epsilon)|}+\bar{F}(M_0)\leq \bar{F}(M_\epsilon) 
\]
where $\bar{F}(M_p)\equiv \sup\{F(M_p,g):\ g\ \mbox{conformal metric on}\ M_p\}$.
\end{proposition}
\begin{proof} For simplicity of notation we assume that $i=1, j=2,k=3$. We see that for $\epsilon$ small, for the metric $\lambda_1$ we have $L_1\sigma_1^{(1o)}\geq c/|\log(\epsilon)|$ since the left hand side is the Dirichlet integral
of the function which is equal to $1$ on $\gamma_1^{(1)}$, $-1$ on $\gamma_2^{(1)}$ and Neumann on the other boundary components. As $\epsilon$
goes to $0$ this function converges to $0$, and thus on a ball around the center of $\gamma_1^{(1)}$ it is a harmonic function which is one on the circle
of radius $\epsilon$ and less than $1/2$ on the circle of radius $r_0$. Such a function has Dirichlet integral at least that of the harmonic function
which is equal to $1$ on $C_\epsilon$ and equal to $1/2$ on $C_{r_0}$ and that is $c/|\log(\epsilon)|$. 

We now show that $\sigma_1^{(1)}=\sigma_1^{(1o)}$ for $\epsilon$ small and for the metric $\lambda_1$. Thus we must show that the lowest nonzero even eigenvalue for $(M_\epsilon,\lambda_1)$ is larger than $\sigma_1^{(1o)}$. 
To see this we construct a new conformal model for $M_p$ related to the $(SN_1)$ eigenvalue problem. We take $t$ to be the harmonic
function which is odd under $\rho_1$ and constant on each component of $\Gamma_1$, satisfies the Neumann condition on each component
 of $\Gamma_i$ for $i\neq 1$, and is normalized so that 
 \begin{equation}  \label{equation:t-normalization}
           \int_{\gamma_1^{(1)}}\nu (t)\ ds=2\pi
 \end{equation}
where $\nu(t)$ denotes the outward normal derivative with respect to the spherical metric. Note that the function $t$ is a scalar multiple of the function
$u_1$ discussed above. We now construct a harmonic conjugate of $t$ which we call $\theta$ and we observe that from the normalization
of $t$ we have the period of $\theta$ around $\gamma_1^{(1)}$ with its natural orientation is $2\pi$. Because of the Neumann condition we see that
$\theta$ has no period about the components of $\Gamma_i$ for $i\neq 1$, and in fact $\theta$ is constant on each of the components. 
Because of the symmetry we can normalize $\theta$ to be $0$ and $\pi$ at the fixed points of $\rho_2$
and $\pi/2$ and $3\pi/2$ at the fixed points of $\rho_3$. This makes the value of $\theta$ to be $0$ and $\pi$ on the components of $\Gamma_2$
and $\pi/2$ and $3\pi/2$ on the components of $\Gamma_3$. The map $(t,\theta)$ now gives a holomorphic map from $M_p$ onto a domain
in the cylinder $\mathbb R\times \mathbb S^1$ which is slit along segments with $\theta=0,\pi/2,\pi,3\pi/2$ with the appropriate symmetry. 

In this model, $(M_\epsilon,\lambda_1)$ is conformally equivalent to the cylinder with the product metric $dt^2+d\theta^2$, and $t$ is a first odd eigenfunction of $(SN_1)$. 
The proof of Proposition \ref{proposition:upper-bound} gives the same bound for the first normalized odd eigenvalue, and by the normalization (\ref{equation:t-normalization}) we have $\sigma_1^{(1,o)}\leq  c/|\log \epsilon|$. 
Let $v$ be a lowest 
(non-constant) even eigenfunction for $(SN_1)$ and suppose its eigenvalue $\sigma$ is less than or equal to that of $t$. 
We may choose $v$ to be normalized to have $L^2$ norm $1$ on $\gamma_1^{(1)}$. Since $v$ is even we have 
\[ \int_{\gamma_1^{(1)}} v(T,\theta)\ d\theta=0
\]
where $\gamma_1^{(1)}=\{t=T\}$. Since $\sigma \leq  \sigma_1^{(1,o)} \leq c/|\log \epsilon|$ it follows that the Dirichlet integral of $v$ is small, so $v$ is almost constant in a
neighborhood of $\gamma_1^{(1)}$. This contradicts the conditions that the $L^2$ norm of $v$ is $1$ and the average value of $v$ is $0$ on
$\gamma_1^{(1)}$. Therefore $t$ is a first eigenfunction and $\sigma_1^{(1)}=\sigma_1^{(1o)}$. 

If we take the metric $g$ which is $\lambda_1$ on $\Gamma_1$ and arbitrary on $\Gamma_2,\Gamma_3$, then we have shown
\[ F(M_\epsilon,g)=a_1^2L_1\sigma_1^{(1o)}+a_2^2L_2\sigma_1^{(2)}+a_3^2L_3\sigma_1^{(3)}.
\]
By the lower bound we obtained and that of Lemma \ref{lemma:epsilon23} we have $F(M_\epsilon, g)\geq c/|\log(\epsilon)|+F(M_0,g)$. Since
$g$ is arbitrary on $M_0$ this implies the corresponding lower bound for the supremum over all conformal metrics on $M_\epsilon$.
\end{proof}

For any $M\in \mathcal M$ we have found a maximizer for each of the three normalized Steklov-Neumann eigenvalue problems among functions
which are odd under the reflection associated with the corresponding boundary pair. We now take up the problem of finding critical points of these
maximizers over the moduli space $\mathcal M$. We have determined a metric $g$ on $M$ which is equal to $\lambda_i$ on $\Gamma_i$
and we evaluate $F(g)$ for the lowest odd eigenfunctions
\[ F^o(g)=\sum_{i=1}^3 a_i^2L_g(\Gamma_i)\sigma_1^{(io)}(g)=\sum_{i=1}^3 a_i^2L_i=\sum_{i=1}^3  a_i^2 \int_{\partial M}u_i\nu u_i=E(u)
\]
where $u=(a_1u_1,a_2u_2,a_3u_3)$ and $E(u)$ is its energy. To summarize, for any $M\in \mathcal M$ we have constructed a harmonic map $u$ from $M$
into the rectangular solid $R=[-a_1,a_1]\times [-a_2,a_2]\times [-a_3,a_3]$ which takes each boundary component of $M$ to a corresponding 
face of $R$. In order to get a minimal surface we need to find critical points of the function $E:\mathcal M\to \mathbb R^+$. By taking the radii
$r_i$ small it is easy to see that energy goes to zero, so there can be no nontrivial minimum. On the other hand by letting one of the radii approach $\pi/2$
and the other two necessarily go to zero, we can see that the energy goes to infinity. Thus there is also no maximum. Our strategy will be to find
saddle points for $E$, and the way we achieve this is to show that the subset of $\mathcal M$ on which $E$ is large consists of at least three components
which are neighborhoods of the boundary points $(\pi/2,0,0)$, $(0,\pi/2,0)$, and $(0,0,\pi/2)$. On the other hand the set of points of $\mathcal M$ on which
$E$ is small consists of a neighborhood of the origin $(0,0,0)$, so the set on which $E$ is greater than a small constant is connected. If there were no
boundary this would imply that $E$ must have critical points. The difficulty is that these expected critical points might lie on the boundary of $\mathcal M$.
 To show that this does not happen, we show that the gradient of $E$ at points near the portion of the boundary where $E$ is bounded points strictly
 into $\mathcal M$.  
 
 We now give the details of these claims. We first need a formula for the partial derivatives of $E(r_1,r_2,r_3)$. We may write $E=a_1^2E_1+a_2^2E_2+a_3^2E_3$
 where $E_i=\int_M\|\nabla u_i\|^2\ da$. We compute $\partial E_i/\partial r_j$ by constructing a vector field $V$ in a neighborhood of $M$ on $\mathbb S^2$
 which is equal to the gradient of the signed distance from $\Gamma_j$ (chosen to be positive inside $M$) at points near $\Gamma_j$ and we extend it arbitrarily inside $M$ with the 
 restriction that $d\rho_k(V)=V$ for $k=1,2,3$. If we let $\Psi_t$ denote the flow of $V$, then we observe that for $t$ small we have 
 $\Psi_t(M_p)=M_{p+te_j}$ where $p=(r_1,r_2,r_3)$, $M_p$ denotes the corresponding surface, and $e_j$ is the $j$th standard basis vector of 
 $\mathbb R^3$. Now we see from the isometric invariance of energy
 \[ E(p+te_j)=E_{\Psi_t(M)}(u_t,g_0)=E_M(u_t\circ\Psi_t,\Psi_t^*(g_0))
 \]
 where $u_t$ denotes the map defined on the surface associated with the point $p+te_j$ and $g_0$ is the standard metric on $\mathbb S^2$. We thus
 see that
 \[ 
    \fder{E}{r_j}=-\int_M \langle T,\mathcal L_V(g_0)\ da=-2\int_{\Gamma_j} T(\nu,V)\ ds_0
 \]
 where $\mathcal L$ is the Lie derivative and $T_{ij}=\tfrac{\partial u}{\partial  x_i}\cdot \tfrac{\partial u}{\partial x_j}-1/2\|\nabla u\|^2\ (g_0)_{ij}$ is the stress energy tensor and we have integrated 
 by parts using the fact that $T$ is divergence free. Here we are taking $\nu$ to be the outer unit normal vector and since $V=-\nu$ on $\Gamma_j$
 we have
 \[  
     \fder{E}{r_j}=2\int_{\Gamma_j} T(\nu,\nu)\ ds_0=\int_{\Gamma_j}(\|\nu(u)\|^2-\|u'\|^2)\ ds_0
 \]
 where we take $u'$ to be the derivative of $u$ with respect to arclength $s_0$. Thus we see from this calculation that 
 \[ 
      \fder{E_i}{r_i}=\int_{\Gamma_i}(\nu u_i)^2\ ds_0,\ \ \ 
      \fder{E_i}{r_j}=-\int_{\Gamma_j} (u_i')^2\ ds_0\ \mbox{for}\ j\neq i.
 \]
 It follows that 
 \[ 
    \fder{E}{r_i}=\int_{\Gamma_i}[a_i^2(\nu u_i)^2-\sum_{j\neq i}a_j^2(u_j')^2]\ ds_0.
 \]
 
 We now discuss the super-level sets $\{E>c\}$ for large enough $c$. We note first that the three subsets $\{p\in\mathcal M:\ r_i>\tfrac{\pi}{4}\}$
 are mutually disjoint since $r_i+r_j<\tfrac{\pi}{2}$ for $i\neq j$. We let $\mathcal M_0$ denote the complement of these three sets
 \[ \mathcal M_0\equiv\{p\in\mathcal M:\ r_i\leq \tfrac{\pi}{4}\ \mbox{for}\ i=1,2,3\}.
 \]
 We show that $E$ is bounded on $\mathcal M_0$.
 \begin{proposition} For $p\in \mathcal M_0$ we have $E(p)\leq c_0$ with $c_0=\tfrac{32\sqrt{2}}{\pi}\sum_{i=1}^3 a_i^2$
 and therefore the set $\{ p\in\mathcal M:\ E(p)>c_0\}$ has at least three connected components.
 \end{proposition}
 
 \begin{proof} From the variational characterization of $u_i$ we have 
 \[ E(u_i)=\min\{\int_{M}\|\nabla\varphi\|^2\ da_0:\ \varphi=1\ \mbox{on}\ \gamma_1^{(i)}, \ \varphi=-1\ \mbox{on}\ \gamma_2^{(i)}\}.
 \]
 Therefore if we let $A(r_i)$ denote the annular region of $\mathbb S^2$ bounded by $\gamma_1^{(i)}$ and $\gamma_2^{(i)}$
 we have
 \[ E(u_i)\leq \min\{\int_{A(r_i)}\|\nabla\varphi\|^2\ da_0:\ \varphi=1 \ \mbox{on}\ \gamma_1^{(i)}, \ \varphi=-1\ \mbox{on}\ \gamma_2^{(i)}\}.
 \]
 Since we can extend any test function $\varphi$ from $A(\tfrac{\pi}{4})$ to the larger annulus $A(r_i)$ by taking it to be constant in the extra region,
 we see
 \[ E(u_i)\leq \min\{\int_{A(\pi/4)}\|\nabla\varphi\|^2\ da_0:\ \varphi=1\ \mbox{on}\ \gamma_1^{(i)}, \ \varphi=-1 \  \mbox{on}\ \gamma_2^{(i)}\}.
 \]
 We choose a function $v$ on $A(r_i)$ which is linear as a function of distance $s$ from the center of $\gamma_1^{(i)}$ so that 
 $\|\nabla v\|^2=16/\pi^2$. Now the area of $A(\tfrac{\pi}{4})$ is equal to $2\sqrt{2}\pi$, so we have the bound
 \[ E(u_i)\leq \frac{32\sqrt{2}}{\pi}.
 \]
 Summing over $i$ then yields the upper bound.

To see that the set $\{p\in\mathcal M:\ E(p)>c_0\}$ intersects each of the sets $\{p\in\mathcal M:\ r_i>\tfrac{\pi}{4}\}$ we need only show that the energy
goes to infinity as $r_i$ approaches $\pi/2$. This is an easy estimate which we omit. 
\end{proof}
 
 We now construct open subsets of $\mathcal M_0$ which are invariant under the gradient flow of $E$ and whose union is $\mathcal M_0$. To
 achieve this we take a small positive number $\delta$ and use it to determine a small positive number $\epsilon$ and three small positive functions 
 $\epsilon_1,\epsilon_2, \epsilon_3$. We define a set $V_\delta\subset \mathcal M_0$ as follows. We first require
 points $p=(r_1,r_2,r_3)\in \mathcal M_0$ to satisfy $r_i+r_j<\pi/2-\epsilon$ ($i\neq j$) in order to lie in $V_\delta$. We have the three continuous
 positive functions $\epsilon_1(r_2,r_3)$, $\epsilon_2(r_1,r_3)$, and $\epsilon_3(r_1,r_2)$ which we will describe below. Our set $V_\delta$ will
 then be given by
 \[ V_\delta=\{p\in\mathcal M_0:\ r_i+r_j<\pi/2-\epsilon\ \mbox{for}\ i\neq j,\ r_1>\epsilon_1(r_2,r_3), r_2>\epsilon_2(r_1,r_3), r_3>\epsilon_3(r_1,r_2)\}.
 \]
 To define the functions $\epsilon_i$, we note that if we fix $r_2,r_3$ and allow $r_1=t$ to vary, the functions $u_1^{(t)}$ are monotone increasing
 in $t$; that is, if $t_1<t_2$, then $u_1^{(t_1)}<u_1^{(t_2)}$ on their common domain $M_{(t_2,r_2,r_3)}$. This property follows from the maximum 
 principle. We see further that as $t$ goes to zero the functions $u_1^{(t)}$ go to zero on compact subsets away from the center of the circle 
 $\gamma_{1,t}^{(1)}$. We now fix a radius $r_0=\pi/8$ and we consider $t=r_1\in (0,r_0)$. For each such $t$ we let $\bar{u}_1(t)$ be the average
 value of $u_1^{(t)}$ over the circle of radius $r_0$. We then see that the function $\bar{u}_1$ is monotone increasing in $t$ and satisfies 
 $\lim_{t\to 0}\bar{u}_1(t)=0$ and $\lim_{t\to r_0}\bar{u}(t)=1$. Thus there is a unique choice of $t$ with $\bar{u}_1(t)=1/2$. We denote this
 choice of $t$ by $h_1(r_2,r_3)$ and we note that since all choices are unique it is a continuous function of $(r_2,r_3)$. 
 We have the following result.
 \begin{lemma} For any $\delta>0$ there is $\epsilon>0$ so that we have $h_1(r_2,r_3)<\delta$ whenever $r_2+r_3>\pi/2-2\epsilon$.
 \end{lemma}
 \begin{proof} We suppose the result were false, so there is $\delta_0>0$ such that for all $n>0$ there exist $r_2^{(n)},r_3^{(n)}$ such that 
 $h_1(r_2^{(n)},r_3^{(n)})\geq\delta_0$ and $r_2^{(n)}+r_3^{(n)}>\pi/2-1/n$. Therefore if we let $p_n=(\delta_0,r_2^{(n)},r_3^{(n)})$ and we let $u_1^{(n)}$ be the corresponding
 function we have the average of $u_1^{(n)}$ over the circle of radius of $r_0$ less than or equal to $1/2$. On the other hand as $n\to\infty$
 the surface $M_{p_n}$ converges to a limit $M_\infty$ which is disconnected and does not intersect the equator parallel to the circles $\Gamma_1^{(n)}$.
 It follows that the limit of $u_1^{(n)}$ is the constant function $1$ and this limit is uniform on compact subsets of $M_\infty$. This contradicts
 the condition on the average value of $u_1^{(n)}$ on the circle of radius $r_0$ and completes the proof of the lemma.
  \end{proof}
  
  We now define $\epsilon_1(r_2,r_3)=\min\{\delta,h_1(r_2,r_3)\}$. Similarly we define the
 functions $\epsilon_2(r_1,r_3)$ and $\epsilon_3(r_1,r_2)$. Note that as our choice of $\delta$ goes to zero, the functions $\epsilon_i$ tend to zero and 
 $\epsilon$ goes to zero, and so $\mathcal M_0=\cup_{\delta>0}V_\delta$.
 
 We now prove the main result that the gradient flow of $E$ preserves the set $V_\delta$ for small enough $\delta$. We let $\Psi_t$ denote
 the gradient flow of $E$ on $\mathcal M$ for any $t$ for which it is defined. 
 \begin{theorem} For $\delta$ sufficiently small and for any $t>0$ we have 
 \[ \Psi_t(V_\delta)\cap\mathcal M_0\subset V_\delta.
 \]
 \end{theorem}
 
 \begin{proof} The set $V_\delta$ in the interior of $\mathcal M_0$ is defined by the six inequalities, so its boundary is the set where one or more of
 the inequalities is an equality. We must show that if any inequality is an equality then the gradient flow makes it a strict inequality. For example
 if $p=(r_1,r_2,r_3)$ satisfies $r_1=\epsilon_1(r_2,r_3)$ then we must show that $\partial E/\partial r_1(p)>0$. To see this we use the formula
 \[ 
     \fder{E}{r_1}(p)=\int_{\Gamma_1}[(\nu u_1)^2-(u_2')^2-(u_3')^2]\ ds=2\int_{\gamma_1^{(1)}}[(\nu u_1)^2-(u_2')^2-(u_3')^2]\ ds.
 \] 
 Since $r_1\leq h_1(r_1,r_2)$ we know that the average of $u_1$ over the circle of radius $r_0=\pi/8$ is less than or equal to $1/2$ while the
 value of $u_i$ on $\gamma_1^{(1)}$ is equal to $1$. To simplify notation we apply a rotationally symmetric conformal transformation from the 
 spherical ball of radius $r_0$ to the unit disk and observe that distances are changed by at most a constant under this transformation. Thus we
 assume that $u_1$ is a harmonic function on the annular region $D_1\setminus D_{\delta_0}$ satisfying $u_1=1$ on $C_{\delta_0}$ and with
 average value $\bar{u}_1(1)$ at most $1/2$ on the unit circle. The radial harmonic function $\bar{u}_1(r)$ is then equal to 
 \[ \bar{u}_1(r)=\bar{u}_1(1)+(1-\bar{u}_1(1))\frac{\log(r)}{\log(\delta_0)}.
 \]
 Therefore we have 
 \[ 
       -\frac{d\bar{u}_1}{dr}(\delta_0)=-\frac{1-\bar{u}_1(1)}{\delta_0\log(\delta_0)}
       \geq \frac{1}{2\delta_0|\log(\delta_0)|}.
 \]
 This then implies that
 \[ \int_{C_{\delta_0}}\left(\frac{\partial u_1}{\partial r}\right)^2\ ds
        \geq \int_{C_{\delta_0}}\left(\frac{d \bar{u}_1}{d r}\right)^2\ ds
        \geq \frac{\pi}{\delta_0|\log(\delta_0)|^2}.
 \]
 Going back to the sphere we see that 
 \[ \int_{\gamma_1^{(1)}}(\nu u_1)^2\ ds\geq \frac{c}{r_1|\log(r_1)|^2}
 \]
 for a positive constant $c$. 
 
 We now show that the contribution from $u_2$ and $u_3$ are much smaller; in fact we will show
 \[ \int_{\gamma_1^{(1)}} (u_i')^2\ ds\leq cr_1^3,\ \mbox{for}\ i=2,3
 \]
 for a constant $c$. To see this we again make a rotationally symmetric conformal transformation of the region between $\gamma_1^{(1)}$ and the circle of radius
 $\pi/8$ with the same center to the unit disk $D_1$ with the disk $D_{\delta_0}$ removed. We may assume that the function $u_i$ is harmonic in
 $D_1\setminus D_{\delta_0}$ with $\partial u_i/\partial r=0$ on $C_{\delta_0}$. We now let $v=\partial u_i/\partial\theta$ and observe that
 $\Delta v=0$ since the vector field $\partial/\partial\theta$ commutes with $\Delta$. We also note that $v_r=\partial/\partial\theta ((u_i)_r)=0$ on
 $C_{\delta_0}$. We now let $f(r)$ be the circular average of $v^2$
 \[ f(r)=\frac{1}{2\pi r}\int_{C_r} v^2\ ds.
 \]
 We compute the first two derivatives of $f$
 \[ f'(r)=\frac{2}{2\pi r}\int _{C_r}v\frac{\partial v}{\partial r}\ ds=\frac{1}{\pi r}\int_{D_r\setminus D_{\delta_0}}\|\nabla v\|^2\ dxdy
 \]
 where we have used the harmonic property of $v$ and the Neumann condition on $C_{\delta_0}$. Multiplying by $r$ and differentiating again
 \[ \frac{1}{r}(rf'(r))'=2\frac{1}{2\pi r}\int_{C_r}\|\nabla v\|^2\ ds\geq 0.
 \]
 Note that the term on the left is the Laplacian of $f(r)$. It follows that $rf'(r)$ is an increasing function of $r$, and since $f'(\delta_0)=0$ we
 see that $f'(r)\geq 0$ for $0\leq r\leq 1$. It then follows that $f(\delta_0)\leq f(1)\leq c$, and so we have
 \[ \int_{C_{\delta_0}} v^2\ ds\leq c\delta_0.
 \]
 Since $u_i'=r^{-1}v$, this implies
 \[ \int_{C_{\delta_0}} (u_i')^2\ ds\leq c\delta_0^3.
 \]
 
 Putting this back on $\mathbb S^2$ we have the desired bound
 \[  \int_{\gamma_1^{(1)}} (u_i')^2\ ds\leq cr_1^3,\ \mbox{for}\ i=2,3,
 \]
  and combining the bounds we have shown that if $r_1=\epsilon_1(r_2,r_3)$ at $p=(r_1,r_2,r_3)$ then
 \[ \frac{\partial E}{\partial r_1}(p)\geq \frac{c}{r_1|\log(r_1)|^2}.
 \]
 A similar argument shows that 
 \[ \frac{\partial E}{\partial r_i}(p)\geq \frac{c}{r_i|\log(r_i)|^2}
 \]
 at a point $p$ where $r_i=\epsilon_i(r_j,r_k)$ where $j<k$ and $\{i,j,k\}=\{1,2,3\}$.
 
 It remains to consider the case when we have a boundary point of $V_\delta$ with $r_i+r_j=\pi/2-\epsilon$ for some $i<j$. For notational
 convenience we assume $i=2$ and $j=3$ so we have a point $p=(r_1,r_2,r_3)\in \partial V_\delta$ with $r_2+r_3=\pi/2-\epsilon$. Note also
 that $r_1\geq \epsilon_1(r_2,r_3)=h_1(r_2,r_3)$ at such a point so we have the average value of $u_1$ on the circle of radius $r_0=\pi/8$
 greater than or equal to $1/2$. We will show that for $\delta$ sufficiently small we have
 \[ \frac{\partial E}{\partial r_i}(p)<0\ \mbox{for}\ i=2,3.
 \]
 This implies that the gradient flow decreases both $r_2$ and $r_3$ and thus moves $p$ strictly interior to $V_\delta$.
 
 For notational convenience we focus on $\Gamma_2$ and we have as above
 \[ \frac{\partial E}{\partial r_2}=2\int_{\gamma_1^{(2)}}[(\nu u_2)^2-(u_1')^2-(u_3')^2]\ ds\leq 2\int_{\gamma_1^{(2)}}[(\nu u_2)^2-(u_1')^2]\ ds.
 \] 
 We will show that the first term is bounded above independent of $\delta$ (note that this term is large when $r_2$ is small, but in this
 case $r_2$ is close to $\pi/4$ when $\delta$ is small). Next we will show that the second term goes to infinity as $\delta$ goes to zero. 
 
 To bound the first term we observe that the annular region in the sphere between the components of $\Gamma_2$ contains $M_p$, and on this region
 we can find a harmonic function $v(s)$ of the distance from the components of $\Gamma_2$ that satisfies $v=1$ on $\gamma_1^{(2)}$ and
 $v=-1$ on $\gamma_2^{(2)}$. Note that the function in the hemisphere $\mathbb S^2_+$ that contains $\gamma_1^{(2)}$ is positive and 
 increases radially from $0$ to $1$. It follows that the outer (to $M$) normal derivative $\nu v\leq 0$ on $\Gamma_i\cap\mathbb S^2_+$ for $i=1,3$.
 We then show by the maximum principle that $u_2\geq v$ in $M\cap\mathbb S^2_+$. To see this assume that $\Omega\subset M\cap\mathbb S^2_+$
 is the set where $u_2-v<0$. Since $u_2-v=0$ both on the equator and on $\gamma_1^{(2)}$ the minimum point must occur at a point of 
 $\Gamma_i\cap \mathbb S^2_+$ and thus its outer normal derivative at such a point would be negative contradicting the boundary condition
 $\nu(u_2-v)\geq 0$ at those points. It then follows from the strong maximum principle that $\Omega$ is empty and $u_2\geq v$ on $M\cap\mathbb S^2_+$.
 
 It follows from the condition $u_2\geq v$ near $\gamma_1^{(2)}$ that $\nu u_2\leq \nu v$ on $\gamma_1^{(2)}$ where both terms are positive. Thus we
 see that $|\nu u_2|$ is uniformly bounded on $\gamma_1^{(2)}$. 
 
 We now show that the second term becomes arbitrarily large when $\delta$ is small. Given a large positive number $B$ we show that there is a
 $\mu>0$ so that if $\delta\leq\mu$, then
 \[ \int_{\gamma_1^{(2)}}(u_1')^2\ ds\geq B.
 \]
 This we prove by contradiction. Suppose we had a sequence $\delta_i$ going to zero such that for each $i$ we have
 \[ \int_{\gamma_{1,i}^{(2)}}[(u_1^{(i)})']^2\ ds< B
 \]
 where the notation is self explanatory. Note that we also have $\epsilon_i$ going to zero, so the circles $\Gamma_2^{(i)}$ and $\Gamma_3^{(i)}$
 converge to $4$ tangent circles of radius $\pi/4$ which disconnect $\mathbb S^2$. We denote by $\mathbb S^2_+$ the hemisphere that contains
 $\gamma_{1,i}^{(1)}$ and by $M_\infty^+$ the limit of the domains $M_i\cap\mathbb S^2_+$. We also recall that the average value of $u_1^{(i)}$ over the circle of radius $r_0=\pi/8$ parallel to the equator of $\mathbb S^2_+$
 is at least $1/2$. Using standard bounds on harmonic functions we may choose a subsequence again denoted by $u_1^{(i)}$ which converges 
 uniformly (up to the boundary) on compact subsets of $M_\infty^+$ away from the center point of $\mathbb S^2_+$ and the $4$ cusp points on the 
 equator. The limiting function $u_\infty$ is then a smooth harmonic function on $M_\infty^+$ which satisfies the Neumann condition on the boundary. Therefore $u_\infty$ is a constant function which by our construction is at least $1/2$. Now if it were true that 
 \[ \int_{\gamma_{1,i}^{(2)}}[(u_1^{(i)})']^2\ ds< B
 \]
 an easy estimate would imply that for points with arclength $s_0$ from $\gamma_{1,i}^{(2)}\cap \partial\mathbb S^2_+$ we would have
 \[ u_1^{(i)}(s_0)\leq \sqrt{s_0B}
 \]
 and this is less than $1/2$ for $s_0$ small. This is in contradiction to the convergence of $u_1^{(i)}$ to $u_\infty$ uniformly away from
 the cusp points. This contradiction completes the proof that 
 \[ \int_{\gamma_1^{(2)}}(u_1')^2\ ds\geq B
 \]
 for $\delta$ sufficiently small. 
 
 Combining these with the bound on the partial derivative we conclude that
 \[ \frac{\partial E}{\partial r_2}(p)<0
 \]
 for $\delta$ sufficiently small. Similarly we can show that for $\delta$ small
 \[ \frac{\partial E}{\partial r_3}(p)<0.
 \]
 Thus we have shown that at a boundary point $p$ of $V_\delta$ with $r_2+r_3=\pi/2-\epsilon$ the gradient flow moves $p$ strictly
 inside $V_\delta$. The analogous proof shows that the flow moves into $V_\delta$ if $r_1+r_3=\pi/2-\epsilon$ or $r_1+r_2=\pi/2-\epsilon$.
 This completes the proof of the theorem.
 \end{proof}
 
 We can now prove the main existence theorem of this section.
 \begin{theorem} For any choice of positive numbers $a_1,a_2,a_3$ there exists at least one free boundary minimal surface of genus $0$
 with $6$ boundary components in the rectangular prism with side lengths $a_1,a_2,a_3$. Moreover each boundary component lies interior to 
 one of the faces of the prism. 
 \end{theorem}
 
 \begin{proof} We have shown that the gradient flow of $E$ preserves the set $V_\delta$. We have also shown that for a large constant $c$ the
 set $\{p\in V_\delta:\ E(p)>c\}$ has at least three connected components. Therefore we can construct a min-max critical point $p_0$ using paths with 
 endpoints in two distinct components. The condition that $\nabla E(p_0)=0$ then implies that the harmonic map $u=(a_1u_1,a_2u_2,a_3u_3)$ is a conformal
 immersion to a free boundary minimal surface in the rectangular prism with side lengths $2a_1,2a_2,2a_3$. Scaling by a factor of $1/2$ then produces 
 the desired free boundary minimal surface. 
 \end{proof}

 \section{Multiplicity and embeddedness results}
 
 In this section we show that for typical choices of $a_1,a_2,a_3$ there are at least two free boundary minimal surfaces of genus $0$ with $6$
 boundary components lying on the faces of the rectangular prism with side lengths $a_1,a_2,a_3$ in $\mathbb R^3$. Without loss of generality we may assume $a_1\leq a_2\leq a_3$, and we state our main multiplicity result.
 
 \begin{theorem} \label{theorem:multiplicity}
 If $a_1< a_3$ then there are at least two free boundary minimal surfaces of genus $0$ with $6$ boundary components,
 each lying on a face of the rectangular prism with side lengths $a_1,a_2,a_3$. In particular, unless the rectangular prism is a cube there are at
 least two solutions of our problem.
 \end{theorem}
 
 We need some preliminaries before giving the proof. We denote by $E(p)$ the total energy
 \[ E(p)=a_1^2E_1(p)+a_2^2E_2(p)+a_3^2E_3(p)
 \]
 for a point $p=(r_1,r_2,r_3)\in \mathcal M$. Note that a permutation of the radii $r_1,r_2,r_3$ yields a surface in the moduli space which is isometric
 to $M_p$ for $p=(r_1,r_2,r_3)$. For $i<j$ we let $\rho_{ij}$ be the permutation which interchanges $i$ and $j$. We now have the following result.
 
 \begin{proposition} \label{proposition:reflection}
 If $a_i\leq a_j$ and $p=(r_1,r_2,r_3)$ with $r_i>r_j$, then we have $E(\rho_{ij}(p))\geq E(p)$. If $a_i<a_j$ then $E(\rho_{ij}(p))> E(p)$.
 \end{proposition}
 
 \begin{proof} We first note that $E_i$ is a strictly increasing function of $r_i$ and a strictly decreasing function of $r_j$ for $j\neq i$. This is because
 if we take $p'\in\mathcal M$ with $r'_i<r_i$ and $r'_k=r_k$ for $k\neq i$ the surface $M_{p'}$ contains $M_p$ and the function $u_i$ can be extended
 to $M_{p'}$ to be a constant equal to $1$ or $-1$ in the components of $M_{p'}\setminus M_p$. It follows that $E_i(p')<E_i(p)$. Similarly if 
 $p'\in\mathcal M$ satisfies $r'_j>r_j$ and $r'_k=r_k$ for $k\neq j$ then the surface $M_p$ contains $M_{p'}$ and if we consider the function $u_i$
 on $M_p$ we can restrict it as a test function to $M_{p'}$ and we get $E_i(p')<E_i(p)$. 
 
 We next observe that $E_i(\rho_{ij}(p))=E_j(p)$ since both are equal to the minimal energy for functions which are $1$ and $-1$ on the
 components of $\gamma_i$ and Neumann on the other boundary components. We then have
 \[ E(\rho_{ij}(p))=a_i^2 E_j(p)+a_j^2 E_i(p)+a_k^2E_k(p)
 \]
 where $\{i,j,k\}=\{1,2,3\}$. Subtracting we thus have
 \[ E(\rho_{ij}(p))-E(p)=(a_j^2-a_i^2)E_i(p)+(a_i^2-a_j^2)E_j(p)=(a_j^2-a_i^2)[E_i(p)-E_i(\rho_{ij}(p))].
 \]
 We now note that since $r_i>r_j$, the $i$th component of $p$ is larger than the $i$th component of $\rho_{ij}(p)$ and the $j$th component of $p$ is smaller than the $j$th component of $\rho_{ij}(p)$. It then follows that $E(\rho_{ij}(p))\geq E(p)$ with strict inequality if $a_i>a_j$. This completes
 the proof.
 \end{proof}
 
 We need one other preliminary result. One might expect that if the $a_i$ are distinct and $p\in\mathcal M$ is a critical point then the $r_i$
 might also be distinct. We do not know if that is true, but we can show that if all the $r_i$ are the same, then all of the $a_i$ are the same.
 To see this we note that if $p$ is a critical point of $E$ for any choice of the $a_i$, it follows that the map $p\to (E_1(p),E_2(p),E_3(p))$
 has vanishing Jacobian determinant at $p$. This is because the critical point condition is
 \[ \sum_{i=1}^3 a_i^2\frac{\partial E_i}{\partial r_j}=0
 \]
 for $j=1,2,3$. This is the condition that the vector $(a_1^2,a_2^2,a_3^2)$ lies in the kernel of the transpose of the Jacobian matrix. Thus the
 Jacobian determinant must be zero at $p$. 
 
 Recall that $\frac{\partial E_i}{\partial r_i}=2\int_{\gamma_1^{(i)}}(\nu u_i)^2\ ds>0$ and 
 $\frac{\partial E_i}{\partial r_j}=-2\int_{\gamma_1^{(i)}}(u_i')^2\ ds<0$ for $j\neq i$. It follows that if $p=(r,r,r)$ with $r\in (0,\pi/4)$ the Jacobian
 matrix $A$ is of the form $a_{ii}=a(r)$ for $i=1,2,3$ and $a_{ij}=-b(r)$ for $i\neq j$ where $a,b$ are positive functions of $r$. We can now
 prove the preliminary result.
 
 \begin{proposition} \label{proposition:equal-r}
 If $p=(r,r,r)$ is a critical point of $E$ for some choice of $a_1,a_2,a_3$, then we must have $a_1=a_2=a_3$.
 \end{proposition}
 
 \begin{proof} As observed above we must have the determinant of the Jacobian matrix $A$ being $0$. If we let $x=a/b$ then we must have
 the vanishing of the determinant of the matrix $B$ with $b_{ii}=x$ and $b_{ij}=-1$ for $i\neq j$. The determinant of the matrix $B$ is
 easily computed as $(x+1)^2(x-2)$. Since $x>0$ we can only have $x=2$ and the null space is spanned by the vector $(1,1,1)$. Thus 
 we must have $a_1=a_2=a_3$ as claimed.
 \end{proof}
 
 We can now prove the main multiplicity theorem, Theorem \ref{theorem:multiplicity}.
 \begin{proof} Assume that $a_1\leq a_2\leq a_3$ with $a_1<a_3$. Assume further that there is only one critical point $p\in\mathcal M$ with
 $p=(r_1,r_2,r_3)$. For $i=1,2,3$, let $Q_i$ be the point in $\overline{\mathcal M}$ with $i$th coordinate $\pi/2$ and the other coordinates $0$.
 Since we have shown that taking a max-min on paths from $Q_i$ to $Q_j$ for any $i<j$ produces a critical point and there is only one critical point
 $P$, it follows that there is a path $\sigma(t)$ in $\overline{\mathcal M}$ parametrized on $[-1,1]$ so that $\sigma(-1)=Q_i$, $\sigma(0)=p$,
 $\sigma(1)=Q_j$, $\sigma(t)\in\mathcal M$ for $t\in (-1,1)$, $E(\sigma(t))$ is uniquely minimized at $t=0$, and the minimum value of $E$ along
 any other path from $Q_i$ to $Q_j$ is at most as large as $E(\sigma(0))=E(p)$. 
 
 We use reflections to obtain information about the relative sizes of $r_1,r_2,r_3$. We first show that if $a_i\leq a_j$ then $r_j\leq r_i$. We prove this by
 contradiction supposing $r_i<r_j$. Since $\sigma(t)$ tends to $Q_i$ as $t$ tends to $-1$ we see that $\sigma_i(t)>\sigma_j(t)$ for $t$ near $-1$.
 We then let $t_0$ be the smallest value of $t$ for which $\sigma_i(t)=\sigma_j(t)$, and so for $t<t_0$ we have $\sigma_i(t)>\sigma_j(t)$. By
 our assumption this inequality does not hold at $t=0$, and so we must have $t_0<0$. We now construct the path $\tilde{\sigma}$ from $Q_i$ 
 to $Q_j$ which is equal to $\sigma$ for $t\in [-1,t_0]$ together with the reflection of the same path under $\rho_{ij}$. Since $\sigma_i>\sigma_j$
 along that segment of $\sigma$ and $a_i\leq a_j$ it follows from Proposition \ref{proposition:reflection} that the energy on the reflected path is larger. On the other hand the minimum of the
 energy for $t\leq t_0$ is strictly larger than $E(p)$, and thus $\tilde{\sigma}$ is a path from $Q_i$ to $Q_j$ on which the minimum of $E$ is
 strictly larger than $E(p)$, a contradiction. Thus we have shown that $r_i\geq r_j$ for $i<j$.
 
 Since we have ordered the $a_i$ it follows from the previous paragraph that $r_3\leq r_2\leq r_1$ and by Proposition \ref{proposition:equal-r} we must have
 $r_3<r_1$ since we are assuming the $a_1<a_3$. We now consider a min-max path $\sigma(t)$ from $Q_2$ to $Q_3$. Since $r_1>r_3$,
 the set of $t$ for which $\sigma_1(t)>\sigma_3(t)$ contains $t=0$. We let the interval $I$ be the connected component of that set which contains 
 $t=0$ and we observe that $I=(t_0,t_1)$ where $t_0\in [-1,0)$ and  $t_1\in (0,1)$. Note that $t_1<1$ because $\sigma_3(t)>\sigma_1(t)$ for
 $t$ near $1$. By continuity we have $\sigma_1=\sigma_3$ at the endpoints of $I$, and we now construct a comparison path 
 $\tilde{\sigma}$ from $Q_2$
 to $Q_3$ by keeping $\sigma(t)$ for $t\notin I$ and replacing the remainder of $\sigma$ with the reflected path 
 $\rho_{13}(\sigma(t))$. Since $\rho_{13}$ fixes $Q_2$ the path $\tilde{\sigma}$ extends continuously from $Q_2$ to $Q_3$. Furthermore since $\sigma_1>\sigma_3$
 in $I$, by Proposition \ref{proposition:reflection}, the values of $E$ along the reflected path are larger. It then follows that the minimum of $E$ along the path $\tilde{\sigma}$ is strictly
 larger than $E(p)$ which is a contradiction to the assumption the $r_1>r_3$. This then contradicts the starting assumption that there was only one
 critical point for $E$. 
 \end{proof}
 
 We now show that the surfaces we construct are embedded. Precisely, we show that for any choices of $(a_1,a_2,a_3)$ and $p\in\mathcal M$, a
 critical point of the associated energy,
 the map $u=(a_1u_1,a_2u_2,a_3u_3)$ is a proper embedding of $M_p$ into $[-a_1,a_1]\times [-a_2,a_2]\times [-a_3,a_3]$ which is smooth
 and embedded up to the boundary. 
 
 We first determine the critical points of $u_i$. Note that if we identify opposite boundary curves of $M_p$ we obtain a closed surface $S_p$
 of genus $3$. We see that the differential $du_i$ is well defined and smooth on $S_p$. From the Hopf boundary point lemma we see that $du_i$
 is nowhere $0$ on $\Gamma_i$. If we consider the function $u_i$ on $\Gamma_j$ for $j\neq i$, we see that it is Neumann and so will have critical
 points at the critical points of the restriction of $u_i$ to $\Gamma_j$. These include at least a maximum point and a minimum point, so this leads to at least $2$ zeroes of $du_i$ on $S_p$ for each $j\neq i$. Thus we have at least $4$ zeroes of $du_i$, and since $u_i$ is harmonic, each zero has negative Hopf index. 
 Since the Euler characteristic of $S_p$ is $-4$ we conclude that there are exactly $4$ zeroes each of index $-1$ occurring at the maximum and 
 minimum points of $u_i$ on $\Gamma_j$ for $j\neq i$. If $k\neq i,j$ the maximum and minimum must occur at the fixed points of $\rho_k$ since each component of $\Gamma_j$ is invariant under $\rho_k$ and $u_i$ is even under $\rho_k$. In particular we have shown that $du_i\neq 0$ at all interior
 points of $M_p$. 
 
 We will show that $u=(a_1u_1,a_2u_2,a_3u_3)$ is an embedding by showing that the image of the upper
 half $M_p^+=\{x\in M_p:\ x_3>0\}$ is a graph $a_3u_3=a_3u_3(a_1u_1,a_2u_2)$. It then follows that $u$ is an embedding of all of $M_p$.
 We note that $u$ is defined on $M_p^+$, and so if we can show that the map $v=(a_1u_1,a_2u_2)$ is a diffeomorphism from $M_p^+$
 to a domain in $\mathbb R^2$ we will have shown the result.
 
 \begin{proposition} \label{proposition:diffeomorphism}
The map $v$ is a diffeomorphism from $M_p^+$ onto a domain of $\mathbb R^2$.
 \end{proposition}
 
 \begin{proof} For $i=1,2,3$ let $X_i$ be the conformal vector field on $\mathbb S^2$ gotten by projecting the standard basis vector field
 $e_i$ to the sphere; thus $X_i(x)=e_i-x_i x$. Since the flow $F_t$ generated by $X_i$ consists of conformal transformations, it follows that
 $u_i\circ  F_t$ is harmonic for each $t$. Thus we see that the derivative $X_iu_i$ is harmonic on $M_p$. We claim that $X_iu_i>0$
 on $M_p$. To see this it suffices to check it on the boundary and then apply the maximum principle. Since we are taking $u_i=-1$
 on the component of $\Gamma_i$ on which $x_i<0$ we see that $X_iu_i>0$ on $\Gamma_i$. Now if $j\neq i$ we have shown that $u_i$
 achieves its minimum at the points of $\Gamma_j$ with smallest $x_i$ and increases up to its maximum at the point of maximum $x_i$.
 Since $u_i$ is Neumann on $\Gamma_j$ we see that $\nabla u_i$ is tangent to $\Gamma_j$ and points in the positive $x_i$ direction.
 Thus we have on $\Gamma_j$, $X_iu_i=X_i\cdot \nabla u_i=e_i\cdot \nabla u_i$ and this is nonnegative and strictly positive except at the
 maximum and minimum points of $u_i$ on $\Gamma_j$. 
 
 We now consider the map $v=(a_1u_1,a_2u_2)$ on $M_p^+$. It will be convenient to write it in complex form $w=a_1u_1+\sqrt{-1}a_2u_2$ and
 we note that $w$ is a complex valued harmonic map. Since $M_p^+$ is of genus $0$, we can choose a complex coordinate $z$
 on $M_p^+$. The harmonic condition is then $w_{z\bar{z}}=0$, so $w_z$ is a holomorphic function. Similarly $\bar{w}_z$ is a holomorphic
 function and we consider the ratio $h=\bar{w}_z/w_z$. Thus $h$ is a meromorphic function and it is independent of the choice of complex coordinate
 $z$. Note that the condition $|h|\leq 1$ is equivalent to the condition that the Jacobian determinant of $v$ is nonnegative. 
 
 The condition that $(a_1u_1,a_2u_2,a_3u_3)$ is conformal can be written $w_z\bar{w}_z+(a_3\partial u_3/\partial z)^2=0$. It follows that both
 $w_z$ and $\bar{w}_z$ are nonzero except at the critical points of $u_3$ and there are $4$ of those on the boundary of $M_p^+$. We choose
 the orientation on $M_p^+$ so that the projection map to the $x_1x_2$-plane is orientation preserving. Now the boundary of $M_p^+$
 consists of the $4$ semicircles in $\Gamma_1^+$ and $\Gamma_2^+$ together with the circle $\gamma_1^{(3)}$ and the four arcs of the
 unit circle in the $x_1x_2$-plane joining the $4$ semicircles. The Jacobian determinant of $v$ is $0$ on the four arcs of the unit circle since the
 evenness of $v$ under $\rho_3$ implies that $v$ is Neumann along those. Similarly the Jacobian determinant is $0$ along $\gamma_1^{(3)}$ and
 thus $|h|=1$ along those portions of $\partial M_p^+$. Let us now consider the circular arcs beginning with $(\gamma_1^{(1)})_+$. On this arc we
 have $u_1=1$, so the image under $v$ is a vertical segment in the image plane. We have shown that $u_2$ is increasing along the arc which is oriented
 in the positive $x_2$ direction. Thus the tangential derivative of $v$ along the arc points in the positive $u_2$ direction. On the other hand since $u_2$
 is Neumann along $(\gamma_1^{(1)})_+$ and $u_1=1$ there, the normal derivative of $v$ along the arc points in the negative $u_1$ direction along
 the arc. Thus we see that the image of the tangential and normal vectors under the differential of $v$ form an orthogonal basis with positive orientation.
 It follows that the Jacobian determinant is positive in the interior of the arc $(\gamma_1^{(1)})_+$ and thus $|h|<1$ there. Similarly we have $|h|<1$
 on the other three arcs of $\Gamma_1^+$ and $\Gamma_2^+$. Note that at the maximum points of the arcs we have $\nabla u_3=0$ and so,
by conformality, either $w_z$ or $\bar{w}_z$ is zero at those points. Our argument shows that $w_z\neq 0$ and $\bar{w}_z= 0$ at these points. 
Therefore $h$ has no poles and we may apply the maximum modulus principle to conclude that $|h|<1$ in the interior of $M_p^+$. 

We have thus shown that the map $v$ is an orientation preserving immersion and an open map from the interior of $M_p^+$ to $\mathbb R^2$.
In order to show that $v$ is an embedding and thus a diffeomorphism onto its image domain, it suffices to show that $v$ is an embedding  on 
$\partial M_p^+$. 

Note that $\partial M_p^+$ consists of two curves, one of which is $\gamma_1^{(3)}$, and the other is the union of the $4$ half circles 
$\Gamma_1^+$
and $\Gamma_2^+$ together with the $4$ unit circle arcs in the $x_1x_2$-plane joining these half circles. We can see that $v$ is an embedding
of $\gamma_1^{(3)}$ because we showed that $u_2$ is decreasing and $u_1>0$ on the portion with $x_1>0$. Thus the image under $v$ of that half
of $\gamma_1^{(3)}$ is a graph over the vertical axis lying in the right half plane. Since the image of the other half is the reflection, we see that $v$
is an embedding on $\gamma_1^{(3)}$. 

To see that $v$ is an embedding of the other boundary curve, we note that because the image is symmetric under reflection through the coordinate
axes and because the sign of $u_i$ is the same as the sign of $x_i$, the image under $v$ of the portion of $M_p^+$ above a quadrant of the
$x_1x_2$-plane lies in the same quadrant. Thus it suffices to show that $v$ is an embedding of the portion of $\partial M_p^+$ with 
$x_1>0$ and $x_2>0$. This curve consists of two components, the first is made up of three segments which we denote $\sigma_1,\sigma_2,\sigma_3$,
and the second denoted $\sigma_4$ is the $1/4$ circle of $\gamma_1^{(3)}$ on which $x_1>0$ and $x_2>0$. To describe the first component,
we take the $1/4$ circle $\gamma_1^{(1)}$ with $x_2>0$ and $x_3>0$, denoted $\sigma_1$, together with an arc of the unit circle
in the $x_1x_2$-plane, denoted $\sigma_2$, joining this $1/4$ circle with the $1/4$ arc of $\gamma_1^{(2)}$ with $x_1>0$ and $x_3>0$, and finally the
$1/4$ arc of $\gamma_1^{(2)}$ which we denote $\sigma_3$. Along $\sigma_1$ we have $v_1=a_1$
and $v_2$ is increasing, so $v$ embeds this arc to a vertical segment $(a_1,t)$ for $0<t<t_0$ for some $t_0>0$. The unit circle arc, $\sigma_2$, is an orbit of the vector field $-X_1$ and so $u_1$ is decreasing and $u_2>0$ on this arc. This implies that the image of $\sigma_2$ is a graph over the 
$v_1$-axis to the left of the previous curve. Finally, on $\sigma_3$, the $1/4$ arc of $\gamma_1^{(2)}$, we have $v_2=a_2$ and $v_1>0$ is
decreasing, so the image of this arc is a segment $(t,a_2)$, $0<t<t_1$ for some $t_1>0$. Since the image $v(\sigma_2)$ lies in the region
$t_1\leq v_1\leq a_1$ we see that $v$ is an embedding on $\sigma_1\cup\sigma_2\cup\sigma_3$. 

To complete the proof we must show that the image of $\sigma_4$ is disjoint from that of the first three arcs so
that $v$ is an embedding on $\sigma_1\cup\sigma_2\cup\sigma_3\cup\sigma_4$. To prove this we use the fact that $v$ is a local diffeomorphism 
with positive 
Jacobian determinant. Since $v_1<a_1$ and $v_2<a_2$ on $\sigma_4$ we clearly have $v(\sigma_4)$ disjoint from $v(\sigma_1)$ and $v(\sigma_3)$. 
To show that $v(\sigma_2)$ and
$v(\sigma_4)$ do not intersect, suppose for the sake of contradiction that $(b_1,b_2)$ is a point of intersection with the largest second coordinate.
Because $v(\sigma_2)$ is a graph over the $v_1$-axis and $v(\sigma_4)$ a graph over the $v_2$-axis, and $b_2$ is the largest second coordinate
of all intersection points we can see that the infinite vertical segment $(b_1,t)$, $b_2<t$ is disjoint from both $v(\sigma_2)$ and $v(\sigma_4)$. On the
other hand the initial points of the segment lie in the image of $v$ because that image lies locally to the right of $v(\sigma_4)$. This would
imply that the entire segment lies in the image of $v$, a contradiction since $v$ is an open map with bounded image. Thus we have shown that $v$
is an embedding on the portion of $\partial M_p^+$ with $x_1>0$ and $x_2>0$. Because of the aforementioned symmetry it follows that $v$
is an embedding on the entire $\partial M_p^+$. It now follows that $v$ is an embedding on $M_p^+$ and therefore a diffeomorphism onto
its image, an open set in $\mathbb R^2$. This completes the proof.
 \end{proof}
 
 We can now prove the following embeddedness theorem.
 \begin{theorem} Suppose $a_1,a_2,a_3$ are positive numbers and that $\Sigma$ is an immersed free boundary minimal surface of genus $0$ in 
 $[-a_1,a_1]\times[-a_2,a_2]\times[-a_3,a_3]$ with one boundary component on each face of the rectangular prism. Assume that $\Sigma$
 is invariant under the reflections in the coordinate planes. Then $\Sigma$ is embedded.
 \end{theorem}
 
 \begin{proof} Because of the reflection symmetries of $\Sigma$, we can conformally parametrize $\Sigma$ by an equivariant conformal 
 harmonic map $u:M_p\to [-a_1,a_1]\times[-a_2,a_2]\times[-a_3,a_3]$ for some choice of $p\in\mathcal{M}$. The point $p$ is then a critical
 point of the corresponding energy functional $E$ on $\mathcal{M}$. As above we write $u=(a_1u_1,a_2u_2,a_3u_3)$.
 
 By Proposition \ref{proposition:diffeomorphism} the map $v=(a_1u_1,a_2u_2)$ is a diffeomorphism from $M_p^+$ to a domain $\Omega$ in the 
 $u_1u_2$-plane. Thus the image under $u$ of $M_p^+$ is a graph over $\Omega$. It follows that $\Sigma=u(M_p)$ is the union of this
 graph with its reflection in the $u_1u_2$-plane and it is an embedded surface. This completes the proof.
 \end{proof}
 
 \section{Maximization and minimal surfaces in products of balls}
 
 In this section we will construct a maximizing metric for $F$ over all metrics with the required symmetries. We do this by first maximizing
 over conformal metrics on the surface $M_p$ for any $p\in\mathcal M$, and then maximizing over $\mathcal M$. In Lemma \ref{lemma:maximization}
 we showed that maximizing $F$ over conformal metrics on $M_p$ is equivalent to maximizing each of the normalized first eigenvalues for the
 Steklov-Neumann problem $(SN_i)$. We showed how to do this for the lowest odd eigenvalue, and we now extend this to maximize
 the first eigenvalue. 

We fix $i\in\{1,2,3\}$ and $p\in\mathcal M$, and consider the maximization of the first eigenvalue of the problem $(SN_i)$ on $M_p$. We recall that $\rho_j$ is the reflection of $\mathbb R^3$ across the coordinate plane orthogonal to the $x_j$-axis. We then consider the class $\mathcal F$ of finite energy maps
$U:M_p\to\mathbb R^3$ that satisfy $U\circ\rho_j=\rho_j\circ U$ for $j=1,2,3$,  with $\|U\|=1$ on $\Gamma_i$. We have the following result.

\begin{proposition} There is a map $U\in\mathcal F$ that minimizes the energy and is smooth up to the boundary of $M_p$ and satisfies the 
boundary conditions $\nu(U)=\lambda U$ on $\Gamma_i$ and $\nu(U)=0$ on $\Gamma_j$ for $j\neq i$ where $\lambda=\nu(\|U\|)$ on $\Gamma_i$.
Furthermore the component functions $U=(u_1,u_2,u_3)$ are first eigenfunctions of $(SN_i)$ for the metric $\lambda$ with eigenvalue $1$. 
\end{proposition}

\begin{proof} The existence and regularity of a minimizing map $U$ follow from the half-harmonic map theory (\cite{S}, \cite{DR}, \cite{LP}). The boundary conditions of
$U$ follow from standard variational considerations; $U$ satisfies the Neumann condition on $\Gamma_j$ for $j\neq i$ because we have imposed
no boundary condition on $\Gamma_j$, and the boundary condition on $\Gamma_i$ is the condition that $\nu(U)$ is parallel to $U$ at each point of
$\Gamma_i$. This condition says that $\nu(U)=\alpha U$ for some function $\alpha$, so taking the dot product with $U$ we find 
\[ \alpha=U\cdot\nu(U)=\frac{1}{2}\nu(\|U\|^2)=\nu(\|U\|)\equiv \lambda.
\]
Thus we see that the component functions $U=(u_1,u_2,u_3)$ are eigenfunctions for the problem $(SN_i)$ for the metric $\lambda$.

To see that $u_1,u_2,u_3$ are first eigenfunctions, we observe that if $u_j$ is not identically $0$ then $u_j$ does not change sign in the
region $M_p\cap \mathbb S^2_{j,+}$ where $\mathbb S^2_{j,+}$ denotes the hemisphere containing $\gamma_1^{(j)}$. By replacing $u_j$ by
its negative if necessary we may assume that $u_j>0$ at some point of $M_p\cap \mathbb S^2_{j,+}$. We then observe that the map $\hat{U}$
that has $u_j$ replaced by the function $\hat{u}_j=|u_j|$ in $M_p\cap \mathbb S^2_{j,+}$ and with $\hat{u}_j\circ\rho_j=-\hat{u}_j$ in the other
hemisphere is also a minimizer and is therefore regular. This implies that $u_j>0$ everywhere in $M_p\cap \mathbb S^2_{j,+}$. By similar reasoning
any first eigenfunction $v$ that satisfies $v\circ\rho_j=-v$ does not change sign in $M_p\cap \mathbb S^2_{j,+}$. By Proposition \ref{proposition:even-odd} there is a nonzero first eigenfunction $v$ that is odd under $\rho_j$ for some $j$, with $v>0$ in $M_p\cap \mathbb S^2_{j,+}$. Suppose that 
$\sigma<1$ is the first eigenvalue. If $u_j$ is not identically zero, then we can use the mixed Steklov and Neumann boundary conditions on
$M_p\cap\mathbb S_{j,+}^2$ to see that
\[ \int_{M_p\cap\mathbb S^2_{j,+}}\langle\nabla v,\nabla u_j\rangle\ da=\sigma\int_{\Gamma_i\cap\mathbb S^2_{j,+}} v u_j\ \lambda\ ds=
\int_{\Gamma_i\cap\mathbb S^2_{j,+}} vu_j \ \lambda\ ds.
\]
Since both $v$ and $u_j$ are positive in $\gamma_1^{(j)}\cap\mathbb S^2_+$ this is a contradiction. 

On the other hand if $u_j\equiv 0$, we must use the condition that the map $U$ is minimizing. We note that the stability condition for the map
$U$ shows that for any $V:M_p\to\mathbb R^3$ with $V\circ\rho_k=\rho_k\circ V$ for $k=1,2,3$ and with $V\cdot U=0$ on $\Gamma_i$ we have
\[ \int_{M_p}\|\nabla V\|^2\ da\geq \int_{\Gamma_i}\|V\|^2\ \lambda\ ds.
\]
This follows from the calculation of the second variation for a family $U_t$ with $\dot{U}=V$ and $\|U_t\|\equiv 1$ on $\Gamma_i$ at $t=0$. We
let $W$ denote the second $t$ derivative of $U_t$ at $t=0$. We have
\[ \frac{1}{2}\frac{d^2}{dt^2}E(U_t)=\int_{M_p}\|\nabla V\|^2\ da+\int_{M_p}\langle\nabla U,\nabla W\rangle\ da.
\]
Using the harmonic property of $U$ and the boundary condition we have
\[ \int_{M_p}\langle\nabla U,\nabla W\rangle\ da=\int_{\Gamma_i}\nu(U)\cdot W\ ds=\int_{\Gamma_i}U\cdot W\ \lambda\ ds.
\] 
Since $U_t\cdot U_t\equiv 1$ on $\Gamma_i$ we have at $t=0$
\[ 0=\frac{1}{2}\frac{d^2}{dt^2}\|U_t\cdot U_t\|^2=\|V\|^2+U\cdot W.
\]
Thus the second variation is given by
\[ \frac{1}{2}\frac{d^2}{dt^2}E(U_t)=\int_{M_p}\|\nabla V\|^2\ da-\int_{\Gamma_i}\|V\|^2\ \lambda\ ds
\]
and this is nonnegative as claimed. 

Now suppose $v$ is a first eigenfunction for $(SN_i)$ with respect to the metric $\lambda$ and with $v\circ\rho_j=-v$, and suppose further that
$u_j\equiv 0$. We can then take the variation $V=ve_j$ where $e_j$ is the $j$th standard basis vector of $\mathbb R^3$. The fact that $v$ is odd 
under $\rho_j$ and even under $\rho_k$ for $k\neq j$ (by Proposition \ref{proposition:even-odd}) then implies that $V\circ \rho_k=\rho_k\circ V$ for $k=1,2,3$. Since $u_j\equiv 0$ we have
$U\cdot V\equiv 0$ on $\Gamma_i$. Thus from the stability of $U$ we have
\[ \int_{M_p}\| \nabla v\|^2\ da\geq \int_{\Gamma_i}v^2\ \lambda\ ds,
\]
and this implies that the lowest eigenvalue is at least one.
\end{proof}

We now show that the metric $\lambda$ maximizes the normalized lowest eigenvalue of $(SN_i)$ over all conformal metrics on $M_p$.
\begin{proposition} \label{proposition:maximizes}
Given any conformal metric $g$ on $M_p$ with the reflection symmetries and with lowest eigenvalue $\sigma_1^{(i)}(g)$ 
for $(SN_i)$ we have
\[ L_g(\Gamma_i)\sigma_1^{(i)}(g)\leq \int_{\Gamma_i}\lambda\ ds.
\]
Moreover equality holds only if $g$ is a constant multiple of $\lambda$ on $\Gamma_i$.
\end{proposition}

\begin{proof} Because of the reflection symmetry of $g$ and the oddness of $u_j$ under $\rho_j$ it follows that $\int_{\Gamma_i}u_j\ ds_g=0$.
Therefore we have
\[ \sigma_1^{(i)}(g)\int_{\Gamma_i}u_j^2\ ds_g\leq \int_{M_p}\|\nabla u_j\|^2\ da.
\]
Summing on $j$ we then have
\[ \sigma_1^{(i)}(g)L_g(\Gamma_i)\leq E(U)=\int_{\Gamma_i} \lambda\ ds
\]
which is the desired inequality. If equality holds we see that each $u_j$ is an eigenfunction for $(SN_i)$ with respect to $g$. Thus we have
for each $j$ on $\Gamma_i$
\[ \nu(u_j)= \sigma_1^{(i)}(g)\lambda_g u_j
\]
where $\lambda_g$ denotes the metric $g$ restricted to $\Gamma_i$. This implies that $ \sigma_1^{(i)}(g)\lambda_g=\lambda$ and completes the
proof.
\end{proof}

We now turn to the problem of maximizing $F$ over the full space of metrics. We fix positive numbers $a_1,a_2,a_3$ and note that for a point
$p\in\mathcal M$ the map $U_p:M_p\to \mathbb R^9$ given by $U_p=(a_1U_{1,p},a_2U_{2,p},a_3U_{3,p})$ where $U_{i,p}$ is the map constructed
above for the surface $M_p$ and the problem $(SN_i)$. The map has energy
\[ E(U_p)=a_1^2E(U_{1,p})+a_2^2E(U_{2,p})+a_3^2E(U_{3,p}).
\]
By Proposition \ref{proposition:maximizes} we see that 
\[ E(U_p)=\max\{F(M_p,g):\ g\ \mbox{conformal metric on } M_p\}.
\]

For notational convenience we denote $E(U_p)$ as $E(p)$. We now show that there is a point $p_0\in\mathcal M$ with 
\[ E(p_0)=\max\{E(p):\ p\in\mathcal M\}.
\]
We know that $E$ is bounded above (see Proposition \ref{proposition:F-supremum}), and we show that $E$ extends continuously to the closure $\overline{\mathcal M}$.

\begin{proposition} The function $E$ is locally Lipschitz on $\mathcal M$ and extends continuously to $\overline{\mathcal M}$. The extended 
function, also denoted $E$, vanishes at the five boundary points $(0,0,0)$, $(\pi/4,\pi/4,\pi/4)$, $(\pi/2,0,0)$, $(0,\pi/2,0)$, and $(0,0,\pi/2)$ and 
is positive at all other points of $\overline{\mathcal M}$.
\end{proposition}

\begin{proof} If $K$ is a compact subset of $\mathcal M$ and $p_1,p_2\in K$ we must show that $|E(p_1)-E(p_2)|\leq c\|p_1-p_2\|$ where 
$c$ depends on $K$. We may choose $\|p_1-p_2\|$ to be the maximum norm. Without loss of generality we may assume that $E(p_1)>E(p_2)$
since we can interchange $p_1$ and $p_2$. We can now construct a diffeomorphism $\Psi:M_{p_1}\to M_{p_2}$ such that for the spherical metric
$g_0$ we have $\|\Psi^*(g_0)-g_0\|\leq c\|p_1-p_2\|$ where the norm on the left is the absolute value of the largest eigenvalue. We also choose $\Psi$
to commute with the reflections $\Psi\circ\rho_i=\rho_i\circ\Psi$. We now consider the function $U_{p_2}\circ\Psi$ on $M_{p_1}$, and we see that
\[ E(U_{p_2}\circ\Psi,g_0)-E(U_{p_2}\circ\Psi,\Psi^*(g_0))\leq c\|p_1-p_2\|.
\]
Since the energy is invariant under isometries, we have $E(U_{p_2}\circ\Psi,\Psi^*(g_0))=E(p_2)$, so we have
\[ E(U_{p_2}\circ\Psi,g_0)\leq E(p_2)+c\|p_1-p_2\|.
\]
By the minimizing property of $U_{p_1}$ this implies $E(p_1)\leq E(p_2)+c\|p_1-p_2\|$ and this gives the Lipschitz property of $E$ on $K$.

Now suppose $p=(r_1,r_2,r_3)$ is a point of $\partial\mathcal M$, so we have each $r_j\geq 0$ and $r_j+r_k\leq \pi/2$ for $j\neq k$. Let's first suppose
that one of the $r_j$ is $\pi/2$, for simplicity of notation assume $r_1=\pi/2$. We then have $r_2=r_3=0$. A neighborhood of $p$ in
$\mathcal M$ then consists of points with $r_1>\pi/2-\epsilon$, $r_2<\epsilon$, and $r_3<\epsilon$ for some small $\epsilon>0$. For small
$\epsilon$ we can see that $E$ can be made arbitrarily small. This is because if $q$ is a point in this neighborhood, then $E(U_{1,q})$ is less than
or equal to the minimum of maps from the annulus gotten by removing $\Gamma_2$ and $\Gamma_3$. Since $r_1$ is close to $\pi/2$, this 
minimum is small since we can construct a map from the circle $\gamma_1^{(1)}$ to $\mathbb S^2$ and extend it to be independent of the distance to
$\gamma_1^{(1)}$. The energy of this is small since $r_1>\pi/2-\epsilon$. Similarly the $E(U_{2,q})$ is small since it is bounded above by 
the minimum energy for maps from the annulus bounded by $\Gamma_2$. This annulus has large conformal modulus and we can construct a 
function which is $1$ on $\gamma_1^{(2)}$ and $-1$ on $\gamma_2^{(2)}$ with small energy. By similar reasoning we see that
$E(U_{3,q})$ is small. Thus $E(q)$ can be made arbitrarily small by choosing $\epsilon$ small. Therefore if we define $E(p)=0$,
then $E$ is continuous at $p$. Similarly we define $E((0,\pi/2,0))=E((0,0,\pi/2))=0$. 

Now if we consider a point $p=(r_1,r_2,r_3)$ in $\partial\mathcal M$ with each $r_i<\pi/2$, then we have a limiting domain $M_p$ in $\mathbb S^2$
which may have cusp points on the boundary if $r_j+r_k=\pi/2$ for some $j\neq k$. If $r_i=0$ for some $i$ we remove the points on the $i$-axis
from $M_p$. Thus $M_p$ is a domain in $\mathbb S^2$ invariant under the reflections $\rho_i$ and with at most $6$ boundary components and whose boundary has at most $12$ cusp points. 

Since $p\in\partial\mathcal M$, at least one of the inequalities must be an equality. We consider various cases and we show that each term $E(U_{i,q})$
converges as $q\in\mathcal M$ converges to $p$ and thus extends continuously to $\partial\mathcal M$. We first consider the case when $r_i=0$
and $r_j+r_k<\pi/2$ for $i,j,k$ distinct. If all $r_j=0$, we can take $E(p)=0$ and this will give a continuous extension. Thus let's suppose that $r_k>0$.
In this case $M_p$ is a domain with smooth boundary and with at most $4$ boundary circles ($2$ if $r_j=0$). We can then see that $E(U_{i,q})$ is 
small for $q\in\mathcal M$ near $p$ while $E(U_{k,q})$ approaches $E(U_{k,p})$ where $U_{k,p}$ is a minimizing map from $M_p$ to $\mathbb B^3$
that commutes with the reflections and satisfies $U_{k,p}(\Gamma_{k,p})\subset \mathbb S^2$.  If $r_j=0$ then we take $E(U_{j,p})=0$ while
if $r_j>0$ we construct a minimizing map $U_{j,p}$ as we did for $U_{k,p}$ and we define 
\[ E(p)=\sum_{l=1}^3 a_l^2E(U_{l,p})
\]
and this gives a continuous extension in these cases. 

In case $r_j+r_k=\pi/2$ for some $j\neq k$ then if $i,j,k$ are distinct we have $E(U_{i,q})$ is small for $q\in\mathcal M$ near $p$ because the domain
$M_p$ is disconnected with $\gamma_1^{(i)}$ and $\gamma_2^{(i)}$ in different components, so we can choose a function that is odd under
$\rho_i$ and even under $\rho_j$ and $\rho_k$ and that is $1$ near $\gamma_1^{(i)}$ and has arbitrarily small energy. Therefore we can take 
$E(U_{i,p})=0$ in this case. We now consider the definition of $E(U_{j,p})$ for $j\neq i$. First note that we can define a map $U_{j,p}$ to be
an energy minimizer from the domain $M_p$ to $\mathbb B^3$ that takes $\Gamma_{j,p}$ to $\mathbb S^2$, and this defines $E(U_{j,p})$.
To see that $E(U_{j,q})$ converges to $E(U_{j,p})$ as $q$ approaches $p$ we observe that any finite energy map from $M_p$ to $\mathbb B^3$
can be approximated by a map that is constant in a small neighborhood of each cusp point on $\partial M_p$. For such maps we clearly have
continuity of the energy, and this suffices to prove continuity of the minimum value.  Thus we have shown that $E$ extends continuously to 
$\overline{\mathcal M}$ as claimed.

We also note that one of the maps $U_{i,p}$ is nontrivial except when $p$ is one of the $5$ points $(0,0,\pi/2)$, $(0,\pi/2,0)$, $(\pi/2,0,0)$,
$(0,0,0)$, and $(\pi/4,\pi/4,\pi/4)$. Note that when $p=(\pi/4,\pi/4,\pi/4)$ we have $r_i+r_j=\pi/2$ for all $i\neq j$, and so the union of $\Gamma_i$
and $\Gamma_j$ disconnects $\mathbb S^2$ for all $i\neq j$. This makes each value $E(U_{i,p})=0$ and therefore $E(p)=0$. At all other boundary
points the domain $M_p$ has at most $2$ connected components. This completes the proof.
\end{proof}

Since $\overline{\mathcal M}$ is compact and $E$ is continuous there, we have a point $p_0=(r_1,r_2,r_3)$ at which $E$ attains its maximum
value. We must show that $p_0$ lies in the interior; that is, $p_0\in\mathcal M$. 
\begin{theorem} There is a point $p_0\in\mathcal M$ such that the surface $M_{p_0}$ maximizes the function $F$ over all metrics with the three
reflection symmetries $\rho_1,\rho_2,\rho_3$. Therefore $M_{p_0}$ can be conformally immersed into $\mathbb B^3\times\mathbb B^3\times\mathbb B^3$
as a free boundary minimal surface.
\end{theorem}

\begin{proof} We must show that a maximum point $p_0$ for $E$ lies in $\mathcal M$. We first observe that $p_0$ cannot be one of the $5$ points
at which $E=0$, so we have $r_i<\pi/2$ for each $i$. We next show that all of the $r_i$ are positive. This follows from 
Proposition \ref{proposition:lower-bound}
since it implies that if $r_i=0$, then for small $\epsilon>0$ the surface $M_\epsilon$ with $r_i=\epsilon$  and the same $r_j$ for $j\neq i$ has energy at least 
$E(p_0)+c/|\log(\epsilon)|$ for a positive constant $c$.  

To complete the proof we must show that $r_j+r_k<\pi/2$ for all $j\neq k$. Suppose we have $r_j+r_k=\pi/2$ so that $\Gamma_j\cup\Gamma_k$
disconnects $\mathbb S^2$. If we denote by $i$ the third index, then we must have $0<r_i\leq \pi/4$ since one of $r_j,r_k$ must be at least $\pi/4$
and the sum of any two is at most $\pi/2$. Let us assume that $r_k\geq \pi/4$, so we have $r_j=\pi/2-r_k\leq \pi/4$, and $r_i\leq \pi/2-r_k$. If we had
$r_i=\pi/2-r_k$, then we would have both $E(U_{i,p_0})=E(U_{j,p_0})=0$, and by shrinking $r_i$ slightly we would increase $E(U_{j,p_0})$ and $E(U_{k,p_0})$
while keeping $E(U_{i,p_0})=0$. Thus we would increase $E$, and so, since $p_0$ is a maximum point, we must have $r_i<\pi/2-r_k$.

We now complete the proof by showing that the surface $M_\epsilon$ with the same $r_i$ and with 
$r_j+r_k=\pi/2-\epsilon$ for small $\epsilon>0$ has $E(M_\epsilon)>E(p_0)+c\sqrt{\epsilon}$ for a positive constant $c$. Since $E(U_{i,q})$ is the
maximum normalized first eigenvalue over conformal metrics on $M_q$, in order to give a lower bound it suffices to give it for the spherical metric.
It is easy to see that for $\epsilon$ small the lowest eigenfunction for $(SN_i)$ is odd under $\rho_i$ since the lowest even eigenvalue does not
go to zero with $\epsilon$. Thus as we let $\epsilon$ go to zero, the first eigenfunctions converge to a positive constant in one component of
$\mathbb S^2\setminus (\Gamma_j\cup\Gamma_k)$ and minus that constant in the other component. By scaling we may assume that the positive
constant is $1$. If we let $u_\epsilon$ be the eigenfunction, then we consider the function $v_\epsilon\equiv (1-2u_\epsilon)_+$ on $\mathbb{S}^2_{i,+}  \cap  M_\epsilon$. We note that 
the function $v_\epsilon$ is supported near the boundary cusp points of the region of $\mathbb S^2\setminus (\Gamma_{j,\epsilon}\cup\Gamma_{k,\epsilon})$ on which $u_\epsilon$ is positive. We give a lower bound on the Dirichlet integral of $v_\epsilon$ by $c\sqrt{\epsilon}$
for some $c>0$. We choose coordinates on a component $V_\epsilon$ of the support of $v_\epsilon$ so that for positive constants $\alpha_1\leq \alpha_2$,
and $\alpha$ with $0<\alpha\leq \alpha_1/2$
\[ V_\epsilon=\{(x,y):\ 0<y<\alpha,\ (x+\alpha_1+\epsilon/2)^2+y^2>\alpha_1^2,\ (x-\alpha_2-\epsilon/2)^2+y^2>\alpha_2^2\}.
\]
Thus $V_\epsilon$ approaches a spherical cusp when $\epsilon$ tends to $0$ with $\alpha$ fixed. We now consider functions $v$ with $v=1$ on 
$\bar{V}_\epsilon\cap\{y=0\}$ and $v=0$ on $\bar{V}_\epsilon\cap\{y=\alpha\}$. We show that the Dirichlet integral of such a function is bounded
below by $c\sqrt{\epsilon}$. If we minimize the Dirichlet integral over such functions, we produce a harmonic function which satisfies the indicated boundary conditions and is Neumann on the remainder of $\partial V_\epsilon$. 

We will derive the lower bound by comparison with a problem that we can explicitly compute. To do such a comparison, we first observe that if
we take a subdomain $W\subset V_\epsilon$ and we consider the same minimization problem on $W$, then its value on $W$ is smaller, since we can
restrict a minimizer for $V_\epsilon$ to $W$ to get a competitor for the problem on $W$ with smaller Dirichlet integral. Similarly if we enlarge $V_\epsilon$ to a domain $W$ that
replaces $\bar{V}_\epsilon\cap\{y=0\}$ by a curve lying in the lower half space $y\leq 0$, then the minimum among functions that are $1$ on the bottom
boundary component of $W$ is smaller than that for $V_\epsilon$
because we can take a minimizer on $V_\epsilon$ and extend it to be $1$ in $W\cap\{y\leq 0\}$ and it becomes a competitor for $W$ with the same Dirichlet integral.

We now consider a domain $V$ that is the exterior of a pair of circles of equal radius $b$ tangent to the $y$-axis at $0$ and with centers at $(b,0)$ and
$(-b,0)$. We restrict attention to the bounded region $V_0=\{(x,y)\in V:\ \sqrt{\epsilon}<y<b/4\}$. We show that the infimum
for functions that are $1$ on the bottom and $0$ on the top is bounded below by $c\sqrt{\epsilon}$ by comparing it with the region described in polar
coordinates $(r,\theta)$ by 
\[ W_0=\{(r,\theta)\in V:\ \sqrt{\epsilon}\sin(\theta)<r<b/2\sin(\theta)\}.
\]
The minimizing function $v$ on $W_0$ is given explicitly by
\[ v=\left(\frac{1}{\sqrt{\epsilon}}-\frac{2}{b}\right)^{-1}\left(\frac{\sin(\theta)}{r}-\frac{2}{b}\right).
\]
An easy estimate shows that $\int_{W_0}\|\nabla v\|^2\ dxdy\geq c\sqrt{\epsilon}$ (since $v$ is harmonic and satisfies either homogeneous Neumann 
or Dirichlet boundary conditions except at the bottom boundary, the value is the integral of the normal derivative of $v$ which is on the order of $1/\sqrt{\epsilon}$ while the length of the lower boundary is of the order of $\epsilon$). Since $V_0\subset W_0$ (note that $b/4$ is smaller than the 
minimum height of $r=b/2\sin(\theta)$ for points of $V$), and the Neumann boundary
components of $V_0$ are contained in those of $W_0$, we can take a minimizer on $V_0$ and extend it to $W_0$ as $0$ above the upper boundary
and $1$ below the lower boundary keeping the Dirichlet integral fixed. Thus the infimum for $V_0$ lies above that of $W_0$ and thus is at least
$c\sqrt{\epsilon}$. 

Finally we must show that a region of the form $V_0$ lies inside $V_\epsilon$ for some choice of $\alpha$ with the upper and lower boundary
components also inside those of $V_\epsilon$. Given a circle of radius $a$ centered at the point $(a+\epsilon/2,0)$, we determine a circle centered
at a point $(a+\delta,-\sqrt{\epsilon})$ of radius $a+\delta$ that contains the given circle and is tangent at some point. The radial segment from 
$(a+\delta,-\sqrt{\epsilon})$ to the point of tangency then passes through $(a+\epsilon/2,0)$. We then have $a+\delta=a+\sqrt{\epsilon+(\epsilon/2-\delta)^2}$  which we can solve as $\delta=1+\epsilon/4$. Thus the circle of radius $a+\delta$ centered at $(a+\delta,-\sqrt{\epsilon})$ is the smallest radius circle tangent to the $y$-axis at $(0, -\sqrt{\epsilon})$ containing the given circle. 

Now given our region $V_\epsilon$ determined by the circles of radius $\alpha_1\leq \alpha_2$, we choose $b=\alpha_2+1+\epsilon/4$ and we form the region $V_0$ using the circles of radius $b$ with the point of tangency being $(0,-\sqrt{\epsilon})$. We translate the region vertically by 
$\sqrt{\epsilon}$ and we see
that this region is contained in $V_\epsilon$ with $\alpha=b/4$, and it serves as a comparison region to show that the minimum Dirichlet integral
for our problem on $V_\epsilon$ is at least $c\sqrt{\epsilon}$. 

We have thus shown that if $r_j+r_k=\pi/2-\epsilon$ then $E_i(M_\epsilon)\geq c\sqrt{\epsilon}$ where $E_i$ denotes the energy of the map
associated with $(SN_i)$. To complete the proof we must show that if $p_0=(r_1,r_2,r_3)$ with $r_j+r_k=\pi/2$ and the other inequalities strict
then we have $E_j(M_\epsilon)\geq E_j(p_0)-c\epsilon$ and $E_k(M_\epsilon)\geq E_k(p_0)-c\epsilon$. To see this we observe that if $M_{q_1}$
and $M_{q_2}$ have the property that there is a diffeomorphism $\Psi:M_{q_1}\to M_{q_2}$ such that $\Psi^*(g_0)$ is near $g_0$ in the sense that
\[ (1-c\epsilon)g_0\leq \Psi^*(g_0)\leq (1+c\epsilon)g_0,
\]
then we have $|E_j({q_1})-E_j({q_2})|\leq c_1\epsilon$ for a constant $c_1$. To apply this we observe that for the point $q_1$ with $r_k$ reduced by $\epsilon$ (and $r_i$, $r_j$ unchanged) and the point $q_2$ with  $r_j$ reduced by $\epsilon$ (and $r_i$, $r_k$ unchanged) the domains $M_{q_1}$ and $M_{q_2}$ are related by such a diffeomorphism $\Psi$.
Since $(SN_j)$ is Neumann on $\Gamma_k$ we have $E_j({q_1})\geq E_j(p_0)$. On the other hand the existence of $\Psi$ tells us that
$E_k({q_1})\geq E_k({q_2})-c\epsilon$. Since $(SN_k)$ is Neumann on $\Gamma_j$ we have $E_k({q_2})\geq E_k(p_0)$. Combining these 
we have $E_k({q_1})\geq E_k(p_0)-c\epsilon$. Putting these together we have
\[ E({q_1})=\sum_{l=1}^3 a_l^2 E_l({q_1})\geq E(p_0)+c\sqrt{\epsilon}.
\]
Since $p_0$ is assumed to be a maximum point of $E$ on $\overline{\mathcal M}$, this contradiction shows that $p_0\in \mathcal M$, and completes the proof.
\end{proof}

\bibliographystyle{plain}

\end{document}